\documentclass{amsart}
\usepackage{amsfonts}
\usepackage{amssymb}
\usepackage{amsmath}

\newtheorem{theorem}{Theorem}[section]
\theoremstyle{plain}

\newtheorem{remark}{Remark}[section]

\numberwithin{equation}{section}
\input{tcilatex}

\begin{document}
\title[EDSs, SQUARES AND\ CUBES ]{ELLIPTIC DIVISIBILITY SEQUENCES, SQUARES
AND\ CUBES}
\author{BET\"{U}L GEZER }
\address{Uludag University, Faculty of Science, Department of Mathematics, G%
\"{o}r\"{u}kle, 16059, Bursa-TURKEY}
\email{betulgezer@uludag.edu.tr}
\maketitle

\begin{abstract}
Elliptic divisibility sequences (EDSs) are generalizations of a class of
integer divisibility sequences called Lucas sequences. There has been much
interest in cases where the terms of Lucas sequences are squares or cubes.
In this work, using the Tate normal form having one parameter of elliptic
curves with torsion points, the general terms and periods of all elliptic
divisibility sequences with a zero term are given in terms of this parameter
by means of Mazur's theorem, and it is shown that which term of $h_{n}$ of
an EDS can be a square or a cube by using the general terms of these
sequences.
\end{abstract}

\begin{flushleft}
{\footnotesize \textbf{AMS Subject Classification 2010: }11B37, 11B39,
11B83, 11G05.\newline
\textbf{Keywords:} }Elliptic curves, torsion points, Tate normal forms,
elliptic divisibility sequences, squares, cubes.

{\footnotesize \textbf{Date:} 21. 09. 2011}
\end{flushleft}

\section{Introduction}

A\textit{\ divisibility sequence} is a sequence $(h_{n})~(n\in {\mathbb{N}})$
of integers with the property that ~$h_{n}|h_{m}~$if $n|m$. There are also
elliptic divisibility sequences satisfying a nonlinear recurrence relation
that comes from the recursion formula for elliptic division polynomials
associated to an elliptic curve.

An\textit{\ elliptic divisibility sequence} (EDS) is a sequence $(h_{n})$ of
integers satisfying a nonlinear recurrence relation%
\begin{equation}
h_{m+n}h_{m-n}=h_{m+1}h_{m-1}h_{n~}^{2}-h_{n+1}h_{n-1}h_{m~}^{2}  \label{1.1}
\end{equation}%
and such that $h_{n}$ divides $h_{m}$ whenever $n$ divides $m$ for all $%
m\geq n\geq 1$. The recurrence relation (\ref{1.1}) is less straightforward
than a linear recurrence.

EDSs are generalizations of a class of integer divisibility sequences called 
\textit{Lucas sequences}. EDSs are interesting because they were the first
nonlinear divisibility sequences to be studied. Morgan Ward wrote several
papers detailing the arithmetic theory of EDSs \cite{MW1, MW2}.

In order to calculate terms, there are two useful formulas (known as \textit{%
duplication formulas}) which are obtained from (\ref{1.1}) by setting first $%
m=n+1$, $n=m$ and then $m=n+1$, $n=m-1$:%
\begin{eqnarray}
~h_{2n+1} &=&h_{n+2}h_{n}^{3}~-h_{n-1}h_{n+1}^{3}\text{,} \\
h_{2n}h_{2} &=&h_{n}(h_{n+2}h_{n-1}^{2}-h_{n-2}h_{n+1}^{2})\text{ }
\end{eqnarray}%
for all $n\in {\mathbb{N}}$. A solution of (\ref{1.1}) is proper if $%
h_{0}=0,h_{1}=1,$ and $h_{2}h_{3}\neq 0.$ Such a proper solution will be an
EDS if and only if \ $h_{2},h_{3},h_{4}$ are integers with $h_{2}|h_{4}$.
The sequence $(h_{n})$ with initial values $h_{0}=0,h_{1}=1,h_{2},h_{3}$ and 
$h_{4}$ is denoted by $[1;$ $h_{2};$ $h_{3};$ $h_{4}]$.

Note that in order to compute the term $h_{n}$ of the elliptic divisibility
sequence $(h_{n})$ by any of the formulas above we need to know all previous
terms before $h_{n}$, but this would be very wasteful.

Ward gave formulas for a very special case of the EDSs whose second or third
term is zero which are called \textit{improper} \textit{sequences}. To over
come this problem we give general terms of the EDSs with zero terms. This
will also help us to determine the square and cube terms in these sequences
as described in the following sections.

One of the well known theorems in the theory of elliptic curves is Mazur's
theorem which states that there is no points of order $11$ and greater than $%
12$ on an elliptic curve over $%
\mathbb{Q}
$. Hence, the rank of the elliptic divisibility sequences associated to
elliptic curves with points of finite order can not be $11$ and greater than 
$12$. Therefore, any of second, ..., tenth and twelfth terms of an elliptic
divisibility sequence can be zero.

In this work we will answer the following questions:

\begin{itemize}
\item What are the initial values of the EDSs with second, ..., tenth and
twelfth term zero?

\item Is there any formula to calculate the terms of the EDSs with second,
..., tenth and twelfth term zero which is more useful than relations above,
i.e., what are the general terms of these sequences?

\item What are the periods of these sequences?

\item Which terms of these sequences can be a square or a cube?
\end{itemize}

The first two questions are discussed in Section 3. The initial values of
the EDSs with zero terms and the general terms of these sequences are given
in Theorem 3.1, and Theorem 3.2, respectively. The third question is
considered in Section 4, and the periods of these sequences are given in
Theorem 4.2. In Section 5, the question of when a term of an elliptic
divisibility sequence with zero terms can be a square or a cube was
discussed in detail.

\section{Some preliminaries on elliptic curves and EDSs}

An \textit{elliptic curve} over $%
\mathbb{Q}
$, is the set of solutions to an equation of the normal form, or generalized
Weierstrass form,%
\begin{equation}
E:y^{2}+a_{1}xy+a_{3}y=x^{3}+a_{2}x^{2}+a_{4}x+a_{6}  \label{01}
\end{equation}%
with coefficients $a_{1},...,a_{6}$ in $%
\mathbb{Q}
$. The set of all solutions $(x,y)\in 
\mathbb{Q}
\times 
\mathbb{Q}
$ to the equation (\ref{01}) together with the point $\mathbf{O}$, called
the \textit{point at infinity}, is denoted by $E(%
\mathbb{Q}
{\mathbb{)}}$ and called the set of $%
\mathbb{Q}
$-\textit{rational points }on $E$. The set of $%
\mathbb{Q}
$-points on $E$ forms an abelian subgroup of $E$ known as the Mordell-Weil
group of $E$. For more details on elliptic curves in general, see \cite%
{JS,JS1}. One of the most important theorems in the theory of elliptic
curves is Mordell-Weil theorem, which implies that, if ${\mathbb{K}}$ is a
number field containing $%
\mathbb{Q}
$, then $E({\mathbb{K)}}$ is a finitely generated abelian group. Also, the
Mordell-Weil theorem shows that $E_{tors}({\mathbb{K)}}$, the \textit{%
torsion subgroup} of $E({\mathbb{K)}}$, is finitely generated and abelian,
hence it is finite, since its generators are of finite order. It is always
interesting to characterize the torsion subgroup of a given elliptic curve.
The question of a uniform bound on $E_{tors}(%
\mathbb{Q}
{\mathbb{)}}$ was studied from the point of view of modular curves by
Shimura, Ogg, and others. In 1976, B. Mazur proved the following strongest
result which had been conjectured by Ogg:

\begin{theorem}
(Mazur\cite{BM}) Let $E$ be an elliptic curve defined over $%
\mathbb{Q}
$. Then the torsion subgroup $E_{tors}(%
\mathbb{Q}
{\mathbb{)}}$ is either isomorphic to $%
\mathbb{Z}
/N%
\mathbb{Z}
$ for $N=1,2,...,10,12$ or to $%
\mathbb{Z}
/2%
\mathbb{Z}
\times 
\mathbb{Z}
/2N%
\mathbb{Z}
$ for $N=1,2,3,4$. Further, each of these groups does occur as an $E_{tors}(%
\mathbb{Q}
{\mathbb{)}}$.
\end{theorem}

One of the aim of this work is to give the general terms of the elliptic
divisibility sequences with this property by using Tate normal form with one
parameter of an elliptic curve according to this parameter.

The \textit{Tate normal form} of an elliptic curve $E$ with point $P=(0,0)$
is defined by 
\begin{equation*}
E:y^{2}+(1-c)xy-by=x^{3}-bx^{2}.
\end{equation*}

If an elliptic curve in normal form has a point of order $N>3$, then
admissible change of variables transforms the curve to the Tate normal form,
in this case the point $P=(0,0)$ is a torsion point of maximal order.
Especially, if we want a classification with respect to the order of the
torsion points, the use of Tate normal form of elliptic curves is
unavoidable.

In \cite{KU}, Kubert listed that there is one parameter family of elliptic
curves $E$ defined over $\mathbb{Q}$ with a torsion point of order $N$ where 
$N=4,...,10,12$. Most cases are proved by Husem\"{o}ller, \cite{HS}. Also
some algorithms are given by using the existence of such a family, \cite{GO}%
. To decide when an elliptic curve defined over $\mathbb{Q}$ has a point of
given order $N$, we need a result on parametrization of torsion structures:

\begin{theorem}
\cite{GO} Every elliptic curve with a point $P$ of order $N=4,...,10,12$ can
be written in the following Tate normal form%
\begin{equation*}
E:y^{2}+(1-c)xy-by=x^{3}-bx^{2},
\end{equation*}%
with the following relations:\newline
1. If $N=4,$ $b=\alpha ,c=0.$\newline
2. If $N=5,$ $b=\alpha ,c=\alpha .$\newline
3. If $N=6,$ $b=\alpha +\alpha ^{2},~c=\alpha .$\newline
4. If $N=7,$ $b=\alpha ^{3}-\alpha ^{2},~c=\alpha ^{2}-\alpha .$\newline
5. If $N=8,$ $b=(2\alpha -1)(\alpha -1),~c=b/\alpha .$\newline
6. If $N=9,c=\alpha ^{2}(\alpha -1),~b=c(\alpha (\alpha -1)+1).$\newline
7. If $N=10,$ $c=(2\alpha ^{3}-3\alpha ^{2}+\alpha )/(\alpha -(\alpha
-1)^{2}),~b=c\alpha ^{2}/(\alpha -(\alpha -1)^{2}).$\newline
8. If $N=12,$ $c=(3\alpha ^{2}-3\alpha +1)(\alpha -2\alpha ^{2})/(\alpha
-1)^{3},~b=c(-2\alpha ^{2}+2\alpha -1)/(\alpha -1).$
\end{theorem}

Theorem 2.2 states that, if any elliptic curve has a point of finite order
then this curve is birationally equivalent to one of the Tate normal forms
given in the theorem above. Therefore, in this work, we are only interested
in the elliptic curves in Tate normal forms with one integer parameter $%
\alpha $ and general terms of the elliptic divisibility sequences are given
depending on this integer parameter $\alpha $.

Ward proved, in \cite{MW2}, that birationally equivalent elliptic curves are
associated to equivalent elliptic divisibility sequences, so it is not a
restriction to give the general terms by using Tate normal forms, that is,
we will be giving general terms of all elliptic divisibility sequences with
zero terms under this equivalence.

The relation between an elliptic curve and an elliptic divisibility sequence
is given by Morgan Ward, see for details and formulas, \cite{MW2}. Ward
proved that elliptic divisibility sequences arise as values of the division
polynomials of an elliptic curve, i.e., if $P=(x,y)$ is a rational point on
an elliptic curve $E$ over $%
\mathbb{Q}
$ then the elliptic divisibility sequence $(h_{n})$ is defined by $%
h_{n}=\psi _{n}(x,y)$ for $n\in {\mathbb{N}}$ where $\psi _{n}$ is the $n$%
-th division polynomial of $E$. Therefore, if $E$ is an elliptic curve over $%
\mathbb{Q}
$ then the initial values of the elliptic divisibility sequence are given by
the coefficients of an elliptic curve. Conversely, if $(h_{n})$ is an
elliptic divisibility sequence in which neither $h_{2}$ nor $h_{3}$ is zero
then there exists an elliptic curve $E$ and the coefficients of the elliptic
curve are given by the initial values of the sequence. In this paper, under
this fact, we first give initial values and the general terms of an elliptic
divisibility sequence associated to an elliptic curve in Tate normal form
with a torsion point $P$. We will now give a short account of material about
elliptic divisibility sequences, for more detailed information about these
sequences in general, see \cite{ME,GE1,RS,CS,MW1,MW2}.

Two elliptic divisibility sequences $(h_{n})$ and $(h_{n}^{\prime })$ are
said to be equivalent if there exists a rational $\omega $ such that 
\begin{equation}
h_{n}^{\prime }=\omega ^{n^{2}-1}h_{n}  \label{d1}
\end{equation}%
for all $n\in \mathbb{N}$.

Since we are interested in sequences with zero terms, i.e., the sequences in
certain ranks, we have to know the concept of the rank of an EDS:

An integer $m$ is said to be a \textit{divisor }of the sequence $(h_{n})$ if
it divides some term $h_{k}$ with $k>0$. Let $m$ be a divisor of $(h_{n})$.
If $\rho $ is an integer such that $m|h_{\rho }$ and there is no integer $j$
such that $j$ is a divisor of $\rho $ with $m|h_{j}$, then $\rho $ is said
to be the \textit{rank of apparition }of $m$ in $(h_{n})$. Ward established
that the multiples of\ $\rho $ are regularly spaced in $(h_{n})$ in the
following theorem.

\begin{theorem}
\cite{MW2}$~$Let $p$ be a prime divisor of an elliptic divisibility sequence 
$(h_{n})$, and let $\rho $ be its smallest rank of apparition. Let $h_{\rho
+1}\not\equiv 0~(p)$. Then 
\begin{equation*}
h_{n}\equiv 0~(p)\text{ if and only if\ }n\equiv 0~(\rho ).
\end{equation*}%
\ \ \ \ \ \ \ 
\end{theorem}

The following theorem shows us that the initial values of the EDS given by
the coefficients of the elliptic curve.

\begin{theorem}
\cite{RS}$~$Let $(h_{n})$ be an elliptic divisibility sequence. Then the
elliptic curves $E:y^{2}+a_{1}xy+a_{3}y=x^{3}+a_{2}x^{2}+a_{4}x$ where $%
a_{1},a_{2},a_{3},a_{4}\in \mathbb{Q},$ associated to $(h_{n})$ are
precisely those with:%
\begin{eqnarray}
h_{2} &=&a_{3},  \label{5} \\
h_{3} &=&a_{2}a_{3}^{2}-a_{4}^{2}-a_{1}a_{3}a_{4}  \label{6} \\
h_{4} &=&2a_{3}a_{4}h_{3}+a_{1}a_{3}^{2}h_{3}-a_{3}^{5}.  \label{7}
\end{eqnarray}
\end{theorem}

\section{The initial values and the general terms of the EDSs}

The problem of finding the general terms of the elliptic divisibility
sequences whose second (resp. third, fourth, fifth, sixth) term is zero are
given in \cite{BO1}. However, it was seen that the general terms of the
other sequences with zero terms can not be easily obtained.

In this paper we first give the general terms of all elliptic divisibility
sequences with zero terms by using Tate normal form of an elliptic curve
which has a torsion point $P=(0,0)$. According to the Mazur's theorem, we
know that if $(h_{n})$ is an elliptic divisibility sequence, in which $%
h_{N}=0$ for some minimal index $N$ then $N\in \{2,...,10,12\}$. The aim of
this section is to give general terms of these sequences. Naturally, it is
sufficient to give the general terms only for the sequences with one of the
terms $h_{2},...,h_{10},h_{12}$ are zero. But we begin by the case $N>3$,
since we use Tate normal form of an elliptic curve. The general terms of
improper elliptic divisibility sequences where $h_{2}$ or $h_{3}$ is equal
to zero will be discussed at the end of this section and general terms of
them will be given in Theorem 3.3. In the following theorem, by using the
Tate normal form of an elliptic curve, the initial values of the sequences
with zero terms are given for $N>3$.

\begin{theorem}
Let $(h_{n})~$be an elliptic divisibility sequence in which $h_{N}=0$ for
some minimal index $N\in \left\{ 4,...,10,12\right\} $. Then the initial
values of $(h_{n})$ with integer parameter $\alpha $ given by the following:%
\newline
1. If $N=4,[1;~-\alpha ;$ $-\alpha ^{3};~\ 0].$\newline
2. If $N=5,[1;~-\alpha ;~-\alpha ^{3};~~\alpha ^{6}].$\newline
3. If $N=6,[1;~-\alpha (\alpha +1);~-\alpha ^{3}(\alpha +1)^{3};~\ \alpha
^{6}(\alpha +1)^{5}].$\newline
4. If $N=7,[1;~-\alpha ^{2}(\alpha -1);~-\alpha ^{6}(\alpha -1)^{3};~\
\alpha ^{11}(\alpha -1)^{6}].$\newline
5. If $N=8,[1;~-\alpha ^{3}\beta ;~-\alpha ^{8}\beta ^{3};~\ \alpha
^{14}\beta ^{6}]$, where $\xi =(\alpha -1)(2\alpha -1).$\newline
6. If $N=9,[1;~-\alpha ^{2}(\alpha -1)\gamma ;~-\alpha ^{6}(\alpha
-1)^{3}\gamma ^{3};~\ \alpha ^{12}(\alpha -1)^{6}\gamma ^{5}]$, where $%
\gamma =\alpha ^{2}-\alpha +1$.\newline
7. If $N=10,[1;~-\alpha ^{3}\delta \zeta ^{4};~-\alpha ^{9}\delta ^{3}\zeta
^{10};~\ \alpha ^{16}\delta ^{6}\zeta ^{19}]$, where $\zeta =\alpha -(\alpha
-1)^{2}$ and $\delta =(\alpha -1)(2\alpha -1).$\newline
8. If $N=12,[1;~-(\alpha -1)^{8}\lambda \theta ;~-(\alpha -1)^{20}\lambda
^{3}\theta ^{3};~~(\alpha -1)^{37}\lambda ^{6}\theta ^{5}]$, where $\lambda
=(3\alpha ^{2}-3\alpha +1)(\alpha -2\alpha ^{2})$ and $\theta =2\alpha
-2\alpha ^{2}-1$.

\begin{proof}
1. We first consider an elliptic curve $E$ with a point $P$ of order $N=4$.
Then by Theorem 2.2, the Tate normal form of $E$ is%
\begin{equation}
E:y^{2}+xy-\alpha y=x^{3}-\alpha x^{2}.  \label{a1}
\end{equation}%
By Theorem 2.4, $E$ is associated to the elliptic sequence $(h_{n})$ and the
initial values of the sequence are 
\begin{equation*}
h_{1}=1,~~h_{2}=-\alpha ,~~h_{3}=-\alpha ^{3},~~h_{4}=0.
\end{equation*}%
It is known that, if $P$ is an integer point, and the coefficients $a_{i}$
of the elliptic curve are integers, then the values $h_{n}$ are integers,
and have the divisibility property, that is, $(h_{n})$ is an EDS. Therefore
coefficients $\alpha $ in (\ref{a1}) must be an integer, since we want to
work with elliptic divisibility sequences.\newline
2. Similarly for $N=5$, we have%
\begin{equation*}
E:y^{2}+(1-\alpha )xy-\alpha y=x^{3}-\alpha x^{2}\text{,}
\end{equation*}%
and so, the initial values of the sequence are%
\begin{equation*}
h_{1}=1,~~h_{2}=-\alpha ,~~h_{3}=-\alpha ^{3},~~h_{4}=\alpha ^{6}.
\end{equation*}%
3. For $N=6$, we have%
\begin{equation*}
E:y^{2}+(1-\alpha )xy-\alpha (\alpha +1)y=x^{3}-\alpha (\alpha +1)x^{2}\text{%
,}
\end{equation*}%
and so, the initial values of the sequence are 
\begin{equation*}
h_{1}=1,~~h_{2}=-\alpha (\alpha +1),~~h_{3}=-\alpha ^{3}(\alpha
+1)^{3},~~h_{4}=\alpha ^{6}(\alpha +1)^{5}.
\end{equation*}%
4. For $N=7$, we have%
\begin{equation*}
E:y^{2}+(1-\alpha ^{2}+\alpha )xy-(\alpha ^{3}-\alpha ^{2})y=x^{3}-(\alpha
^{3}-\alpha ^{2})x^{2}\text{,}
\end{equation*}%
and so, the initial values of the sequence are 
\begin{equation*}
h_{1}=1,~~h_{2}=-\alpha ^{2}(\alpha -1),~~h_{3}=-\alpha ^{6}(\alpha
-1)^{3},~~h_{4}=\alpha ^{11}(\alpha -1)^{6}.
\end{equation*}%
5. Now let $E$ be an elliptic curve in normal form with a point $P$ of order 
$N=8$. Then by Theorem 2.2, the Tate normal form of $E$ is%
\begin{equation*}
E:y^{2}+(1-(2\alpha -1)(\alpha -1)/\alpha )xy-(2\alpha -1)(\alpha
-1)y=x^{3}-(2\alpha -1)(\alpha -1)x^{2}.
\end{equation*}%
Since we work with elliptic divisibility sequences, coefficients of the
elliptic curve must be integer, so we transform the elliptic curve $E$ to a
birationally equivalent curve $\widetilde{E}$ under admissible change of
variables%
\begin{equation*}
\widetilde{E}:y^{2}+(\alpha -\xi )xy-\alpha ^{3}\xi y=x^{3}-\alpha ^{2}\xi
x^{2}
\end{equation*}%
where $\xi =(2\alpha -1)(\alpha -1)$. By Theorem 2.4, $\widetilde{E}$ is
associated to the elliptic divisibility sequence $(h_{n})$ and the initial
values of the sequence are 
\begin{equation*}
h_{1}=1,h_{2}=-\alpha ^{3}(\alpha -1)(2\alpha -1),h_{3}=-\alpha ^{8}(\alpha
-1)^{3}(2\alpha -1)^{3},h_{4}=\alpha ^{14}(\alpha -1)^{6}(2\alpha -1)^{6}.
\end{equation*}%
6. For $N=9$ we have%
\begin{equation*}
E:y^{2}+(1-c)xy-by=x^{3}-bx^{2}
\end{equation*}%
where $c=\alpha ^{2}(\alpha -1)$, $b=c(\alpha (\alpha -1)+1)$, and so, the
initial values of this sequence are 
\begin{equation*}
h_{1}=1,~~h_{2}=-\alpha ^{2}(\alpha -1)\gamma ,~~h_{3}=-\alpha ^{6}(\alpha
-1)^{3}\gamma ^{3},~~h_{4}=\alpha ^{12}(\alpha -1)^{6}\gamma ^{5}
\end{equation*}%
where $\gamma =\alpha ^{2}-\alpha +1$.\newline
7. Now let $E$ be an elliptic curve in normal form with a point $P$ of order 
$N=10$. By Theorem 2.2, the Tate normal form of $E$ is%
\begin{equation*}
E:y^{2}+(1-c)xy-by=x^{3}-bx^{2}
\end{equation*}%
where $c=\alpha (2\alpha ^{2}-3\alpha +1)/(\alpha -(\alpha -1)^{2})$, $%
b=c\alpha ^{2}/(\alpha -(\alpha -1)^{2})$ and $E$ birationally equivalent to
the curve $\widetilde{E}$ under admissible change of variables given by%
\begin{equation*}
\widetilde{E}:y^{2}+(\delta ^{2}-\delta \zeta )xy-\alpha ^{2}\zeta \delta
^{4}y=x^{3}-\zeta \alpha ^{2}x^{2}\text{,}
\end{equation*}%
where $\zeta =(2\alpha -1)(\alpha -1)$, $\delta =\alpha -(\alpha -1)^{2}$.
By Theorem 2.4, $\widetilde{E}$ is associated to the elliptic divisibility
sequence $(h_{n})$ and the initial values of this sequence are%
\begin{equation*}
h_{1}=1,~h_{2}=-\alpha ^{3}\zeta \delta ^{4},~h_{3}=-\alpha ^{9}\zeta
^{3}\delta ^{10},~h_{4}=\alpha ^{16}\zeta ^{6}\delta ^{19}.
\end{equation*}%
8. Now let $E$ be an elliptic curve in normal form with a point $P$ of order 
$N=12$. By Theorem 2.2, the Tate normal form of $E$ is%
\begin{equation*}
E:y^{2}+(1-c)xy-by=x^{3}-bx^{2}\text{,}
\end{equation*}%
where $c=(3\alpha ^{2}-3\alpha +1)\alpha (1-2\alpha )/(\alpha -1)^{3}$, $%
b=c(2\alpha -2\alpha ^{2}-1)/(\alpha -1)$ and $E$ birationally equivalent to
the curve $\widetilde{E}$ under admissible change of variables given by%
\begin{equation*}
\widetilde{E}:y^{2}+(\alpha -1)((\alpha -1)^{3}-\lambda )xy-(\alpha
-1)^{8}\lambda \theta y=x^{3}-(\alpha -1)^{4}\lambda \theta x^{2}\text{,}
\end{equation*}%
where $\lambda =(3\alpha ^{2}-3\alpha +1)(\alpha -2\alpha ^{2})$, $\theta
=2\alpha -2\alpha ^{2}-1$. By Theorem 2.4, $\widetilde{E}$ is associated to
the elliptic divisibility sequence $(h_{n})$ and the initial values of this
sequence are 
\begin{equation*}
h_{1}=1,~~h_{2}=-(\alpha -1)^{8}\lambda \theta ,~~h_{3}=-(\alpha
-1)^{20}\lambda ^{3}\theta ^{3},~~h_{4}=(\alpha -1)^{37}\lambda ^{6}\theta
^{5}.
\end{equation*}
\end{proof}
\end{theorem}

Thus we know the initial values of the sequences $(h_{n})$ with zero terms.
We now give the general terms of the sequences $(h_{n})$ with zero terms
depending on only one integer parameter $\alpha $ in the following theorem.

\begin{theorem}
Let $(h_{n})~$be an elliptic divisibility sequence with $N$-th term zero,
i.e., with rank $N\in \left\{ 4,...,10,12\right\} $. Let $\xi ,\gamma
,\delta ,\zeta ,\lambda ,\theta $ as in Theorem 3.1. Then the general term
of $(h_{n})$ given by the following formulas:\newline
1. If $N=4,$ 
\begin{equation}
h_{n}=\varepsilon \alpha ^{\{(3n^{2}-p)/8\}}  \label{8}
\end{equation}%
where $\varepsilon =\left\{ 
\begin{array}{ll}
{\small +1} & \text{if }{\small n}\text{ }{\small \equiv 1,5,6~(8)}\text{ \
\ }{\small \ } \\ 
{\small -1} & \text{if }{\small n}\text{ }{\small \equiv 2,3,7~(8)}\text{,}%
\end{array}%
\right. p=\left\{ 
\begin{array}{ll}
{\small 3} & \text{if }{\small n}\text{ }{\small \equiv 1,3~(4)}\text{ \ \ }%
{\small \ } \\ 
{\small 4} & \text{if }{\small n}\text{ }{\small \equiv 2~(4)}\text{.}%
\end{array}%
\right. $\newline
2. If $N=5,$%
\begin{equation}
h_{n}=\varepsilon \alpha ^{\{(2n^{2}-p)/5\}}  \label{81}
\end{equation}%
where $\varepsilon =\left\{ 
\begin{array}{ll}
{\small +1} & \text{if }{\small n}\text{ }{\small \equiv 1,4,7,8~(10)}\text{
\ \ }{\small \ } \\ 
{\small -1} & \text{if }{\small n}\text{ }{\small \equiv 2,3,6,9~(10)}\text{,%
}%
\end{array}%
\right. p=\left\{ 
\begin{array}{ll}
{\small 2} & \text{if }{\small n}\text{ }{\small \equiv 1,4~(5)}\text{ \ \ }%
{\small \ } \\ 
{\small 3} & \text{if }{\small n}\text{ }{\small \equiv 2,3~(5)}\text{.}%
\end{array}%
\right. $\newline
3. If $N=6,$%
\begin{equation}
h_{n}=\varepsilon \alpha ^{\{(5n^{2}-p)/12\}}(\alpha +1)^{\{(n^{2}-k)/3\}}
\label{n6}
\end{equation}%
where $\varepsilon =\left\{ 
\begin{array}{ll}
{\small +1} & \text{if }{\small n}\text{ }{\small \equiv 1,4,5,9,10~(12)}%
\text{ \ \ }{\small \ } \\ 
{\small -1} & \text{if }{\small n}\text{ }{\small \equiv 2,3,7,8,11~(12),}%
\end{array}%
\right. $and%
\begin{equation*}
~p=\left\{ 
\begin{array}{ll}
{\small 5} & \text{if }{\small n}\text{ }{\small \equiv 1,5~(6)} \\ 
{\small 8} & \text{if }{\small n}\text{ }{\small \equiv 2,4~(6)} \\ 
{\small 9} & \text{if }{\small n}\text{ }{\small \equiv 3~(6),}%
\end{array}%
\text{ }\right. k=\left\{ 
\begin{array}{ll}
{\small 1} & \text{if }{\small n}\text{ }{\small \equiv 1,2,4,5~(6)}\text{ \
\ }{\small \ } \\ 
{\small 0} & \text{if }{\small n}\text{ }{\small \equiv 3(6)}\text{.}%
\end{array}%
\right. \newline
\end{equation*}%
\newline
4. If $N=7,$%
\begin{equation}
h_{n}=\varepsilon \alpha ^{\{(5n^{2}-p)/7\}}(\alpha -1)^{\{(3n^{2}-q)/7\}}
\label{n7}
\end{equation}%
where $\varepsilon =\left\{ 
\begin{array}{ll}
{\small +1} & \text{if }{\small n}\text{ }{\small \equiv 1,4,5~(7)}\text{ \
\ }{\small \ } \\ 
{\small -1} & \text{if }{\small n}\text{ }{\small \equiv 2,3,6~(7)}\text{,}%
\end{array}%
\right. $and%
\begin{equation*}
p=\left\{ 
\begin{array}{ll}
{\small 5} & \text{if }{\small n}\text{ }{\small \equiv 1,6~(7)} \\ 
{\small 6} & \text{if }{\small n}\text{ }{\small \equiv 2,5~(7)} \\ 
{\small 3} & \text{if }{\small n}\text{ }{\small \equiv 3,4~(7),}%
\end{array}%
\right. q=\left\{ 
\begin{array}{ll}
{\small 3} & \text{if }{\small n}\text{ }{\small \equiv 1,6~(7)} \\ 
{\small 5} & \text{if }{\small n}\text{ }{\small \equiv 2,5~(7)} \\ 
{\small 6} & \text{if }{\small n}\text{ }{\small \equiv 3,4~(7).}%
\end{array}%
\right.
\end{equation*}%
\newline
5. If $N=8,$%
\begin{equation}
h_{n}=\varepsilon \alpha ^{\{(15n^{2}-p)/16\}}(\alpha
-1)^{\{(7n^{2}-q)/16\}}(2\alpha -1)^{\{(3n^{2}-k)/8\}}  \label{n8}
\end{equation}%
where $\varepsilon =\left\{ 
\begin{array}{ll}
{\small +1} & \text{if }{\small n}\text{ }{\small \equiv
1,4,5,9,10,13,14~(16)}\text{ \ \ }{\small \ } \\ 
{\small -1} & \text{if }{\small n}\text{ }{\small \equiv
2,3,6,7,11,12,15~(16)}\text{,}%
\end{array}%
\right. $and 
\begin{equation*}
{\small p=}\left\{ 
\begin{array}{ll}
{\small 15} & \text{if }{\small n}\text{ }{\small \equiv 1,7~(8)} \\ 
{\small 12} & \text{if }{\small n}\text{ }{\small \equiv 2,6~(8)} \\ 
{\small 7} & \text{if }{\small n}\text{ }{\small \equiv 3,5~(8)} \\ 
{\small 16} & \text{if }{\small n}\text{ }{\small \equiv 4~(8),}%
\end{array}%
\right. {\small q=}\left\{ 
\begin{array}{ll}
{\small 7} & \text{if }{\small n}\text{ }{\small \equiv 1,7~(8)} \\ 
{\small 12} & \text{if }{\small n}\text{ }{\small \equiv 2,6~(8)} \\ 
{\small 15} & \text{if }{\small n}\text{ }{\small \equiv 3,5~(8)} \\ 
{\small 16} & \text{if }{\small n}\text{ }{\small \equiv 4~(8),}%
\end{array}%
\right. {\small k=}\left\{ 
\begin{array}{ll}
{\small 3} & \text{if }{\small n}\text{ }{\small \equiv 1,3,5,7~(8)} \\ 
{\small 4} & \text{if }{\small n}\text{ }{\small \equiv 2,6~(8)} \\ 
{\small 0} & \text{if }{\small n}\text{ }{\small \equiv 4~(8).}%
\end{array}%
\right.
\end{equation*}%
\newline
6. If $N=9,$%
\begin{equation}
h_{n}=\varepsilon \alpha ^{\{(7n^{2}-p)/9\}}(\alpha
-1)^{\{(4n^{2}-q)/9\}}\gamma ^{\{(n^{2}-k)/3\}}  \label{n9}
\end{equation}%
where $\varepsilon =\left\{ 
\begin{array}{ll}
{\small +1} & \text{if }{\small n}\text{ }{\small \equiv
1,4,5,8,11,12,15,16~(18)}\text{ \ \ }{\small \ } \\ 
{\small -1} & \text{if }{\small n}\text{ }{\small \equiv
2,3,6,7,10,13,14,17~(18)}\text{,}%
\end{array}%
\right. $and 
\begin{equation*}
{\small p=}\left\{ 
\begin{array}{ll}
{\small 7} & \text{if }{\small n}\text{ }{\small \equiv 1,8~(9)} \\ 
{\small 10} & \text{if }{\small n}\text{ }{\small \equiv 2,7~(9)} \\ 
{\small 9} & \text{if }{\small n}\text{ }{\small \equiv 3,6~(9)} \\ 
{\small 4} & \text{if }{\small n}\text{ }{\small \equiv 4,5~(9),}%
\end{array}%
\right. {\small q=}\left\{ 
\begin{array}{ll}
{\small 4} & \text{if }{\small n}\text{ }{\small \equiv 1,8~(9)} \\ 
{\small 7} & \text{if }{\small n}\text{ }{\small \equiv 2,7~(9)} \\ 
{\small 9} & \text{if }{\small n}\text{ }{\small \equiv 3,6~(9)} \\ 
{\small 10} & \text{if }{\small n}\text{ }{\small \equiv 4,5~(9),}%
\end{array}%
\right. {\small k=}\left\{ 
\begin{array}{ll}
{\small 0} & \text{if }{\small n}\text{ }{\small \equiv 3,6~(9)} \\ 
{\small 1} & {\small otherwise.}%
\end{array}%
\right.
\end{equation*}%
\newline
7. If $N=10,$%
\begin{equation}
h_{n}=\varepsilon \alpha ^{\{(21n^{2}-p)/20\}}(\alpha
-1)^{\{(9n^{2}-q)/20\}}(2\alpha -1)^{\{(2n^{2}-k)/5\}}\delta
^{\{(5n^{2}-s)/4\}}  \label{n10}
\end{equation}%
where $\varepsilon =\left\{ 
\begin{array}{ll}
{\small +1} & \text{if }{\small n}\text{ }{\small \equiv
1,4,5,8,9,13,14,17,18~(20)}\text{ \ \ }{\small \ } \\ 
{\small -1} & \text{if }{\small n}\text{ }{\small \equiv
2,3,6,7,11,12,15,16,19~(20)}\text{,}%
\end{array}%
\right. $and%
\begin{eqnarray*}
{\small p} &{\small =}&\left\{ 
\begin{array}{ll}
{\small 21} & \text{if }{\small n}\text{ }{\small \equiv 1,9~(10)} \\ 
{\small 24} & \text{if }{\small n}\text{ }{\small \equiv 2,8~(10)} \\ 
{\small 9} & \text{if }{\small n}\text{ }{\small \equiv 3,7~(10)} \\ 
{\small 16} & \text{if }{\small n}\text{ }{\small \equiv 4,6~(10)} \\ 
{\small 25} & \text{if }{\small n}\text{ }{\small \equiv 5~(10),}%
\end{array}%
\right. {\small q=}\left\{ 
\begin{array}{ll}
{\small 9} & \text{if }{\small n}\text{ }{\small \equiv 1,9~(10)} \\ 
{\small 16} & \text{if }{\small n}\text{ }{\small \equiv 2,8~(10)} \\ 
{\small 21} & \text{if }{\small n}\text{ }{\small \equiv 3,7~(10)} \\ 
{\small 24} & \text{if }{\small n}\text{ }{\small \equiv 4,6~(10)} \\ 
{\small 25} & \text{if }{\small n}\text{ }{\small \equiv 5~(10),}%
\end{array}%
\right. \\
{\small k} &{\small =}&\left\{ 
\begin{array}{ll}
{\small 2} & \text{if }{\small n}\text{ }{\small \equiv 1,4,6,9~(10)} \\ 
{\small 3} & \text{if }{\small n}\text{ }{\small \equiv 2,3,7,8~(10)} \\ 
{\small 0} & \text{if }{\small n}\text{ }{\small \equiv 5~(10),}%
\end{array}%
\right. {\small s=}\left\{ 
\begin{array}{ll}
{\small 5} & \text{if }{\small n}\text{ }{\small \equiv 1,3,5,7,9~(10)} \\ 
{\small 4} & \text{if }{\small n}\text{ }{\small \equiv 2,4,6,8~(10).}%
\end{array}%
\right.
\end{eqnarray*}%
\newline
8. If $N=12,$%
\begin{equation}
h_{n}=\varepsilon \alpha ^{\{(n^{2}-p)/12\}}(\alpha
-1)^{\{(59n^{2}-q)/24\}}(2\alpha -1)^{\{(n^{2}-k)/24\}}\lambda
^{\{(3n^{2}-s)/8\}}\theta ^{\{(n^{2}-t)/3\}}  \label{n12}
\end{equation}%
where $\varepsilon =%
\begin{array}{ll}
{\small +1} & \text{if }{\small n}\text{ }{\small \equiv
1,5,9,13,14,16,17,18,20,21,22~(24)}\text{ \ \ }{\small \ } \\ 
{\small -1} & \text{if }{\small n}\text{ }{\small \equiv
2,3,4,6,7,8,10,11,15,19,23~(24)}\text{,}%
\end{array}%
$and%
\begin{equation*}
{\small p=}\left\{ 
\begin{array}{ll}
{\small 1} & \text{if }{\small n}\text{ }{\small \equiv 1,11~(12)} \\ 
{\small 4} & \text{if }{\small n}\text{ }{\small \equiv 2,10~(12)} \\ 
{\small 9} & \text{if }{\small n}\text{ }{\small \equiv 3,9~(12)} \\ 
{\small 16} & \text{if }{\small n}\text{ }{\small \equiv 4,8~(12)} \\ 
{\small 13} & \text{if }{\small n}\text{ }{\small \equiv 5,7~(12)} \\ 
{\small 12} & \text{if }{\small n}\text{ }{\small \equiv 6~(12),}%
\end{array}%
\right. {\small q=}\left\{ 
\begin{array}{ll}
{\small 59} & \text{if }{\small n}\text{ }{\small \equiv 1,11~(12)} \\ 
{\small 44} & \text{if }{\small n}\text{ }{\small \equiv 2,10~(12)} \\ 
{\small 51} & \text{if }{\small n}\text{ }{\small \equiv 3,9~(12)} \\ 
{\small 56} & \text{if }{\small n}\text{ }{\small \equiv 4,8~(12)} \\ 
{\small 35} & \text{if }{\small n}\text{ }{\small \equiv 5,7~(12)} \\ 
{\small 60} & \text{if }{\small n}\text{ }{\small \equiv 6~(12),}%
\end{array}%
\right. {\small k=}\left\{ 
\begin{array}{ll}
{\small 1} & \text{if }{\small n}\text{ }{\small \equiv 1,5,7,11~(12)} \\ 
{\small 4} & \text{if }{\small n}\text{ }{\small \equiv 2,10~(12)} \\ 
{\small 9} & \text{if }{\small n}\text{ }{\small \equiv 3,9~(12)} \\ 
{\small 16} & \text{if }{\small n}\text{ }{\small \equiv 4,8~(12)} \\ 
{\small 12} & \text{if }{\small n}\text{ }{\small \equiv ~6~(12),}%
\end{array}%
\right.
\end{equation*}%
\begin{equation*}
{\small s=}\left\{ 
\begin{array}{ll}
{\small 3} & \text{if }{\small n}\text{ }{\small \equiv 1,3,5,7,9,11~(12)}
\\ 
{\small 4} & \text{if }{\small n}\text{ }{\small \equiv 2,6,10~(12)} \\ 
{\small 0} & \text{if }{\small n}\text{ }{\small \equiv 4,8~(12),}%
\end{array}%
\right. {\small t=}\left\{ 
\begin{array}{ll}
{\small 1} & \text{if }{\small n}\text{ }{\small \equiv
1,2,4,5,7,8,10,11~(12)} \\ 
{\small 0} & \text{if }{\small n}\text{ }{\small \equiv 3,6,9~(12).}%
\end{array}%
\right.
\end{equation*}
\end{theorem}

\begin{proof}
1. It is clear that the result is true for $n=5$. Hence we assume that $n>5$%
. If $(h_{n})$ is an EDS, then we know that%
\begin{equation}
h_{n+2}h_{n-2}=h_{n+1}h_{n-1}h_{2~}^{2}-h_{3}h_{1}h_{n~}^{2}.  \label{91}
\end{equation}

We argue by induction on $n$. \ First suppose that $n\equiv 1~(4)$ and (\ref%
{8}) is true for $n+1$. Then we have 
\begin{equation*}
h_{n+2}=-\alpha ^{6m^{2}+9m+3}
\end{equation*}%
by (\ref{8}). On the other hand we see that%
\begin{eqnarray*}
h_{n-2} &=&-\alpha ^{6m^{2}-3m}, \\
h_{n} &=&\alpha ^{6m^{2}+3m}, \\
h_{n-1} &=&0.
\end{eqnarray*}%
Substituting these expressions into (\ref{91}) gives $h_{n+2}=-\alpha
_{2}^{6m^{2}+9m+3}.$ Thus we proved the equation (\ref{8}) is true for $n+2$
which completes the proof for $n\equiv 1~(4)$. Other cases can be proved in
the same way. \newline
2. It is clear that the result is true for $n=6$. Hence we assume that $n>6$%
. If $(h_{n})$ is an EDS, then we know that%
\begin{equation}
h_{n+2}h_{n-2}=h_{n+1}h_{n-1}h_{2~}^{2}-h_{3}h_{1}h_{n~}^{2}.  \label{82}
\end{equation}

We again argue by induction using (\ref{81}). First suppose that $n\equiv
1~(5)$ and (\ref{81}) is true for $n+1$. Then we have 
\begin{equation*}
h_{n+2}=\left\{ 
\begin{array}{ll}
-\alpha ^{10m^{2}+12m+3} & \text{if }m\equiv 2,4~(5) \\ 
\alpha ^{10m^{2}+12m+3} & \text{if }m\equiv 1,3~(5)%
\end{array}%
\right.
\end{equation*}%
by (\ref{81}). On the other hand we see that%
\begin{eqnarray*}
h_{n-2} &=&\left\{ 
\begin{array}{ll}
-\alpha ^{10m^{2}-4m} & \text{if }m\equiv 2,4~(5) \\ 
\alpha ^{10m^{2}-4m} & \text{if }m\equiv 1,3~(5)%
\end{array}%
\right. \\
h_{n} &=&\left\{ 
\begin{array}{ll}
\alpha ^{10m^{2}+4m} & \text{if }m\equiv 2,4~(5) \\ 
-\alpha ^{10m^{2}+4m} & \text{if }m\equiv 1,3~(5)%
\end{array}%
\right. \\
h_{n-1} &=&0.
\end{eqnarray*}%
Substituting these expressions into (\ref{82}) gives $h_{n+2}h_{n-2}=\alpha
^{3}h_{n}^{2},$ hence we have%
\begin{equation*}
h_{n+2}=\left\{ 
\begin{array}{ll}
-\alpha ^{10m^{2}+12m+3} & \text{if }m\equiv 2,4~(5) \\ 
\alpha ^{10m^{2}+12m+3} & \text{if }m\equiv 1,3~(5)\text{.}%
\end{array}%
\right.
\end{equation*}%
Thus we proved that the equation (\ref{81}) is true for $n+2.$ Other cases
can be proved in the same way.

The same proof works for the remaining parts of the theorem.
\end{proof}

For example, in the following table, it is seen that the values $p$, $q$, $k$%
, $s$ and $t$ of the general term of the elliptic divisibility sequences
with rank twelve:

\begin{table}[th]
\label{eqtable} \renewcommand\arraystretch{0.9} \noindent 
\begin{equation*}
\begin{array}{|c|c|c|c|c|c|}
\hline
& {p} & {q} & {k} & {s} & {t} \\ \hline\hline
n\equiv 1(12) & 1 & 59 & 1 & 3 & 1 \\ \hline
n\equiv 2(12) & 4 & 44 & 4 & 4 & 1 \\ \hline
n\equiv 3(12) & 9 & 51 & 9 & 3 & 0 \\ \hline
n\equiv 4(12) & 16 & 56 & 16 & 0 & 1 \\ \hline
n\equiv 5(12) & 13 & 35 & 1 & 3 & 1 \\ \hline
n\equiv 6(12) & \mathbf{12} & \mathbf{60} & \mathbf{-12} & \mathbf{4} & 
\mathbf{0} \\ \hline
n\equiv 7(12) & 13 & 35 & 1 & 3 & 1 \\ \hline
n\equiv 8(12) & 16 & 56 & 16 & 0 & 1 \\ \hline
n\equiv 9(12) & 9 & 51 & 9 & 3 & 0 \\ \hline
n\equiv 10(12) & 4 & 44 & 4 & 4 & 1 \\ \hline
n\equiv 11(12) & 1 & 59 & 1 & 3 & 1 \\ \hline
\end{array}%
\end{equation*}%
\par
\begin{center}
{\footnotesize The values $p,q,k,s$ and $t$ of the general term of the EDSs
with rank twelve}
\end{center}
\end{table}
According to this table, we have the following terms for $n\equiv 6$ $(12)$: 
\begin{eqnarray*}
h_{6} &=&\alpha ^{2}(\alpha -1)^{86}(2\alpha -1)^{2}\lambda ^{13}\delta
^{12}, \\
~h_{18} &=&-\alpha ^{26}(\alpha -1)^{794}(2\alpha -1)^{14}\lambda
^{121}\delta ^{108} \\
h_{30} &=&\alpha ^{74}(\alpha -1)^{2210}(2\alpha -1)^{38}\lambda
^{337}\delta ^{300} \\
h_{42} &=&-\alpha ^{146}(\alpha -1)^{4334}(2\alpha -1)^{74}\lambda
^{661}\delta ^{588}
\end{eqnarray*}%
In addition, if we take $\alpha =3$, the first eight terms of the EDS with
rank twelve are%
\begin{eqnarray*}
&&{\small 1;-948480};{\small %
-53329136320512000;-27346122891266847865307136000000;} \\
&&{\small 17500141386070121786711926566237801283584;-3319445579395304657} \\
&&{\small 0787963710047756557989077368249771884544000000000000000;-45938} \\
&&{\small 2422798666425039100328482290559063213427355998727370793872053} \\
&&{\small 34329698182758400000000000000000000;-1884349228191035614337748} \\
&&{\small 21043991142309132528994842811514985817859497226098776607161543} \\
&&{\small 9393817547813525913600000000000000000000000000;...}
\end{eqnarray*}

\begin{remark}
There are also elliptic curves with a torsion point which are not in the
Tate normal form as in Theorem 2.2. For example the point $P=(0,0)$ on the
elliptic curve%
\begin{equation*}
E:y^{2}+17xy-120y=x^{3}-60x^{2}
\end{equation*}%
is a torsion point of order eight. The initial values of elliptic
divisibility sequence $(h_{n})$ associated to the curve $E$ are%
\begin{equation*}
h_{1}=1,~h_{2}=-120,~h_{3}=-864000,~h_{4}=-186624000000
\end{equation*}%
and $h_{8}=0$, that is, the sequence has rank eight. $E$ is birationally
equivalent to the curve $\widetilde{E}$ under the transformation $\dbinom{x}{%
y}\rightarrow \dbinom{4x}{8y}$ given by%
\begin{equation*}
\widetilde{E}:y^{2}+\frac{17}{2}xy-15y=x^{3}-15x^{2}
\end{equation*}%
which is in Tate normal form. This curve gives us an elliptic sequence, so
we need to make another transformation to have an elliptic divisibility
sequence. Hence we have%
\begin{equation*}
E^{\prime }:y^{2}+34xy-960y=x^{3}-240x^{2}
\end{equation*}%
and the initial values of elliptic divisibility sequence $(h_{n}^{\prime })$
associated to the elliptic curve $E^{\prime }$ are 
\begin{equation*}
h_{1}^{\prime }=1,~h_{2}^{\prime }=-960,~h_{3}^{\prime
}=-221184000,~h_{4}^{\prime }=-6115295232000000.
\end{equation*}%
It can easily be seen that $(h_{n})$ and $(h_{n}^{\prime })$ are equivalent
by taking $\omega =-2$ in the equation (\ref{d1}). So, the general terms of
the elliptic divisibility sequences associated to the elliptic curves in
Tate normal form are the general terms of all elliptic divisibility
sequences with zero terms under the equivalence.
\end{remark}

\begin{remark}
There is no Tate normal form of an elliptic curve with the torsion point of
order two or three, but Kubert in \cite{KU}, listed the elliptic curves with
torsion point of order two or three are 
\begin{equation*}
E:y^{2}=x^{3}+ax^{2}+bx
\end{equation*}%
and%
\begin{equation*}
E:y^{2}+a_{1}xy+a_{3}y=x^{3},
\end{equation*}%
respectively. In this case, the elliptic divisibility sequences associated
to an elliptic curve with the torsion point of order two or three give
improper sequences and the initial values of these sequences are 
\begin{equation*}
h_{1}=1,~h_{2}=0,~h_{3}=-b^{2},~h_{4}=0
\end{equation*}%
and 
\begin{equation*}
h_{1}=1,~h_{2}=a_{3},~h_{3}=0,~h_{4}=-a_{3}^{5},
\end{equation*}%
respectively.
\end{remark}

Under these considerations, an easy computation gives the general terms of
the improper divisibility sequences.

\begin{theorem}
\textbf{i.} Let $(h_{n})~$be an elliptic divisibility sequence $%
[1;~0;~-b^{2};~0]$. Then the general term of $(h_{n})$ given by the
following formula:%
\begin{equation*}
h_{n}=\varepsilon b^{\{(n^{2}-1)/4\}}
\end{equation*}%
where%
\begin{equation*}
\varepsilon =\left\{ 
\begin{array}{ll}
+1 & \text{if }n\text{ }\equiv 1,5~(8)\text{ \ \ }\  \\ 
-1 & \text{if }n\text{ }\equiv 3,7~(8)\text{.}%
\end{array}%
\right.
\end{equation*}%
\textbf{ii.} Let $(h_{n})~$be an elliptic divisibility sequence $[1;$ ~$%
a_{3};~~0;~-a_{3}^{5}]$. Then the general term of $(h_{n})$ given by the
following formula:%
\begin{equation*}
h_{n}=\varepsilon a_{3}^{\{(n^{2}-1)/3\}}
\end{equation*}%
where%
\begin{equation*}
\varepsilon =\left\{ 
\begin{array}{ll}
+1 & \text{if }n\text{ }\equiv 1,2~(6)\text{ \ \ }\  \\ 
-1 & \text{if }n\text{ }\equiv 4,5~(6)\text{.}%
\end{array}%
\right.
\end{equation*}
\end{theorem}

\section{The periods of the EDSs}

In this section we will give the periods of all elliptic divisibility
sequences with zero terms by using general terms of these sequences which
are given in previous section.

A sequence $(s_{n})$ of rational integers is said to be \textit{numerically
periodic} \textit{modulo} $m$ if there exists a positive integer $\pi $ such
that 
\begin{equation}
s_{n+\pi }\equiv s_{n}~(m)  \label{a4}
\end{equation}%
for all sufficiently large $n$. If (\ref{a4}) holds for all $n$, then$%
~(s_{n})$ is said to be \textit{purely periodic} \textit{modulo} $m$. The
smallest $\pi $ for which (\ref{a4}) is true is called the \textit{period }%
of $(s_{n})$ modulo $m$. All other $\pi $'s are multiples of it.

The following theorem of Ward shows us how the period and rank are connected.

\begin{theorem}
\cite{MW2} Let $(h_{n})$ be an EDS and $p$ be an odd prime whose rank of
apparition $\rho $ is greater than $3$. Let $a_{1}$ be an integral solution
of the congruence $a_{1}\equiv h_{2}/h_{\rho -2}~(p)$ and let $e$ and $k$ be
the exponents to which $a_{1}$ and $a_{2}\equiv h_{\rho -1}$ $(p)$
respectively belong modulo $p$. Then $(h_{n})$ is purely periodic modulo $p$%
, and its period $\pi $ is given by the formula $\pi (h_{n})=\tau \rho $
where $\tau =2^{\nu }[e,~k].$ Here $[e,~k]$ is the least common multiple of $%
e$ and $k,$ and the exponent $\nu $ is determined as follows:%
\begin{equation*}
\nu =\left\{ 
\begin{array}{ll}
+1 & \text{ if }e\text{ and }k\text{ are both odd \ } \\ 
-1 & 
\begin{array}{l}
\text{if }e\text{ and }k\text{ are both even and both divisible by} \\ 
\text{ \ \ \ \ \ \ exactly the same power of }2%
\end{array}
\\ 
~\text{\ }0 & \text{ otherwise.}%
\end{array}%
\right.
\end{equation*}
\end{theorem}

We give the period of $(h_{n})$ in the following theorem.

\begin{theorem}
\bigskip Let $(h_{n})~$be an elliptic divisibility sequence with $N$-th term
zero, where $N\in \{4,5,...,10,12\}$ and let $p$ be an odd prime. Then the
period of $(h_{n})$ is 
\begin{equation*}
\pi (h_{n})=\left\{ 
\begin{array}{ll}
t(p-1) & \text{if }q=[e,k]~\text{is a primitive root modulo }p \\ 
2Nl & \text{otherwise}%
\end{array}%
\right.
\end{equation*}%
where $l=\left\{ 
\begin{array}{ll}
q & \text{if }q~\text{is odd} \\ 
q/2\text{ } & \text{if }q~\text{is even,}%
\end{array}%
\right. t=\left\{ 
\begin{array}{ll}
N & \text{if }N~\text{is even} \\ 
N/2\text{ } & \text{if }N~\text{is odd.}%
\end{array}%
\right. $
\end{theorem}

\begin{proof}
The cases $N=4$ and $5$, can easily be seen, so we give $N=6$, other cases
can be proved in similar way. In this case the period of $(h_{n})$ is $\pi
(h_{n})=6(p-1)$ or $12l$. By Theorem 3.1 and Theorem 4.1, we have $%
a_{1}=h_{2}/h_{4}$ $=-1/\alpha ^{5}(\alpha +1)^{4}$ and $a_{2}=h_{5}=\alpha
^{10}(\alpha +1)^{8}$. Let $e$ and $k$ be the orders of $a_{1}$ and $a_{2}$,
respectively. Then $k=e/2$ when $e$ is even, and $k=e$ when $e$ is odd,
since $a_{2}=1/a_{1}^{2}$. Let $a_{1}$ be a primitive root modulo $p$. Then $%
e=p-1$, $k=(p-1)/2$ and so $q=p-1$. Hence $\nu =0$ and so $\tau $ $=p-1$.
Therefore in this case $\pi (h_{n})=6(p-1)$. If $a_{1}~$is not a primitive
root modulo $p$, then there are two cases. In the first case, let $q$ be
odd. Then $e=$ $k=q$, so that $\nu =1.$ Thus $\tau $ $=q,$ therefore $\pi
(h_{n})=6q$. In the second case, let $q$ be even. Then $e=q$ and $k=q/2$, so
that $\nu =0.$ Thus $\tau $ $=2q$ therefore $\pi (h_{n})=12q.$
\end{proof}

\begin{table}[th]
\label{eqtable} \renewcommand\arraystretch{0.9} \noindent%
\begin{equation*}
\begin{array}{|c|c|c|c|c|c|c|c|c|c|c|c|}
\hline
N=6 & \mathbb{F}_{5} & \mathbb{F}_{7} & \mathbb{F}_{11} & \mathbb{F}_{13} & 
\mathbb{F}_{17} & \mathbb{F}_{19} & \mathbb{F}_{23} & \mathbb{F}_{29} & 
\mathbb{F}_{31} & \mathbb{F}_{37} & \mathbb{F}_{41} \\ \hline
\alpha=-5 & - & 36 & 60 & 72 & 96 & 108 & 132 & 84 & 180 & 72 & 120 \\ \hline
\alpha=-4 & 12 & 12 & 60 & 36 & 12 & 108 & 132 & 84 & 180 & 36 & 12 \\ \hline
\alpha=-3 & 24 & 36 & 60 & 12 & 96 & 36 & 12 & 168 & 180 & 108 & 240 \\ 
\hline
\alpha=-2 & 24 & 36 & 12 & 72 & 48 & 108 & 132 & 168 & 12 & 216 & 24 \\ 
\hline
\alpha=1 & 12 & 36 & 60 & 36 & 12 & 108 & 132 & 84 & 60 & 108 & 60 \\ \hline
\alpha=2 & 24 & 36 & 60 & 24 & 48 & 36 & 132 & 168 & 180 & 216 & 24 \\ \hline
\alpha=3 & 24 & 12 & 60 & 36 & 96 & 108 & 132 & 168 & 180 & 36 & 240 \\ 
\hline
\alpha=4 & - & 36 & 60 & 36 & 12 & 108 & 132 & 12 & 36 & 36 & 60 \\ \hline
\alpha=5 & - & 36 & 60 & 72 & 96 & 108 & 132 & 84 & 36 & 216 & 120 \\ \hline
\end{array}%
\end{equation*}%
\par
\begin{center}
{\footnotesize The periods of the EDSs with rank six for some integer $%
\alpha $ parameter}
\end{center}
\end{table}
The periods of the elliptic divisibility sequences for which the sixth term
is zero modulo $p$ for some integer parameter $\alpha $ and $p>3$, appears
in the tablo above.

\section{Squares and cubes in EDSs}

The question of when a term of a Lucas sequence can be square has generated
interest in the literature \cite{AB1,AB2,PR1,PR2}. Similar results
concerning cubes were also obtained for specific sequences such as
Fibonacci, Lucas and Pell numbers \cite{AP, PR}. In \cite{BO1,BO2}, we
describe when a term of an elliptic divisibility sequence can be a square or
a cube, if one of the first six terms is zero.

The ultimate purpose of this section is to determine square or cube terms in
some special family of the elliptic divisibility sequences whose contain a
zero term. In this section we determine square or cube terms of these
sequences by using the general terms of them. In particular, we will
investigate the answers of the following questions:

\begin{itemize}
\item \textit{Which terms of }$(h_{n})$\textit{\ can be a square or a cube
independent of }$\alpha $\textit{\ }? This question is answered for each
case. For example, consider an elliptic divisibility sequence for which
sixth term is zero, \newline
\textit{i}. if $n\equiv 1$, $5$, $7$, $11~(12)$,$~$then $h_{n}=\square $ for
all $\alpha \in 
\mathbb{Z}
\backslash \{-1$, $0\}$,\newline
\textit{ii. }if $n\equiv 1$, $3$, $9$, $15$, $17~(18)$,$~$then $h_{n}=C$ for
all $\alpha \in 
\mathbb{Z}
\backslash \{-1$, $0\}$.

\item \textit{Which terms of }$(h_{n})$\textit{\ can not be a square or a
cube }? Starting with the fact that square or cube terms can be arise
dependent on the parameter $\alpha $ it is seen that some terms of $(h_{n})$%
\textit{\ }can not be a square or a cube for any choice of $\alpha $ for
each case. For example, consider an elliptic divisibility sequence for which
sixth term is zero, \newline
\textit{i}. if $n\equiv 2$, $3$, $9$, $10~(12)$, then $h_{n}$ is not a
square for all $\alpha \in 
\mathbb{Z}
\backslash \{-1$, $0\}$, \newline
\textit{ii}. if $n\equiv 2$, $5$, $7$, $11$, $13$, $16~(18)$, then $h_{n}$
is not a cube for all $\alpha \in 
\mathbb{Z}
\backslash \{-1,0\}$.

\item \textit{Which terms of }$(h_{n})$\textit{\ can be a square or a cube
with admissible choice of }$\alpha $ ? In addition to square or cube terms
which determined in question one it is seen that a term of an EDS can be a
square or a cube depending on the admissible choice of $\alpha $. For
example, consider an elliptic divisibility sequence for which sixth term is
zero, \newline
\textit{i}. if $n\equiv 4$, $8~(12)$ then $h_{n}$ is a square iff $\alpha
+1=\square $, \newline
\textit{ii}. if $n\equiv 4$, $14~(18)$ then $h_{n}$ is a cube iff $\alpha
+1=C$, \newline
\textit{iii}. if $n\equiv 8$, $10~(18)$ then $h_{n}$ is a cube iff $\alpha
=C $.
\end{itemize}

Especially when we look for the answers of our problems we deduce some
equations whose solutions give the desired answers. The problem now reduces
to establishing that Diophantine equations which are divided into five main
classes, then finding all solutions of these equations and observing whether
these solutions give the desired $\alpha $. It follows that observing the
solutions of some Pell equations, classical equations, trivial equations and
finding the integral points on elliptic curves with rank zero or one. Now we
have many equations of similar type and so we make a table to bring together
all of these equations and their solutions can observe from this chart. For
example, consider an elliptic divisibility sequence for which the eighth
term is zero. If $n\equiv 3$, $13~(16)$,$~$then we see that $h_{n}=\square $
iff 
\begin{equation*}
(\alpha -1)(2\alpha -1)=\square
\end{equation*}%
this leads to a Pell's equation $(4\alpha -3)^{2}-8\beta ^{2}=1$ or a
trivial equation $(4\alpha -3)^{2}+8\beta ^{2}=1$ where $\alpha $, $\beta $
are integers. Equations encountered in some cases turned into elliptic
curves. In particular, if we have an elliptic curve with rank zero then the
only integral points on this curve are the torsion points. These, in turn,
can be computed \ by the Lutz-Nagell Theorem. If the elliptic curve has a
rank different from zero then the \textit{Elliptic Logarithm Method} is
applied to find the all integral solutions. Throughout this paper the
symbols $\square $ and $C$ mean a square and a cube of a non-zero rational
number, i.e. $\square =\pm \beta ^{2}$ where $\beta $ is an integer.

A basic observation is the following: For every equation, the distinct
irreducible factors (over $%
\mathbb{Q}
\lbrack \alpha ]$) appearing in the left-hand side (if they are at least
two) are pairwise relatively prime.$^{1}$ This implies that, if the right
hand-side is $\square $ (respectively, $C$), then every irreducible factor
is $\square $ (respectively, $C$). We use of this fact, for a quite number
of equations, it turns out that it is not necessary to make use of this fact
for all irreducible factor of the left-hand side.$^{2}$

\begin{equation*}
\text{Solutions of the equations}
\end{equation*}%
\newline
$%
\begin{tabular}{||c|c|c|c||}
\hline\hline
\textbf{Eq. No} & 
\begin{tabular}{c}
\textbf{implies} \\ 
\textbf{or} \\ 
\textbf{is equivalent to}%
\end{tabular}
& 
\begin{tabular}{c}
\textbf{is reduced to} \\ 
\textbf{or} \\ 
\textbf{is equivalent to}%
\end{tabular}
& \textbf{Comments} \\ \hline
{\small 1} & ${\small \alpha (\alpha +1)=\square }$ & ${\small (2\alpha +1)}%
^{2}{\small \pm \beta }^{2}{\small =1}$ & {\small trivial eq.} \\ \hline
{\small 5} & ${\small \alpha (\alpha -1)=\square }$ & ${\small (2\alpha -1)}%
^{2}{\small \pm \beta }^{2}{\small =1}$ & {\small trivial eq.} \\ \hline
{\small 2, 3, 4} & ${\small \alpha =\beta }_{1}^{3}${\small \ \& }${\small %
\alpha +1=\beta }_{2}^{3}$ & ${\small \beta }_{2}^{3}{\small -\beta }_{1}^{3}%
{\small =1}$ & {\small trivial eq.} \\ \hline
\begin{tabular}{c}
{\small 6, 10, 15, } \\ 
{\small 16,17, 18, } \\ 
{\small 19, 29,30, } \\ 
{\small 38, 39, 40} \\ 
{\small 41, 42}%
\end{tabular}
& 
\begin{tabular}{c}
${\small \alpha =\beta }_{1}^{3}${\small \ }${\small \&}${\small \ }${\small %
\alpha -1=\beta }_{2}^{3}$%
\end{tabular}
& ${\small \beta }_{1}^{3}{\small -\beta }_{2}^{3}{\small =1}$ & {\small %
trivial eq.} \\ \hline
{\small 7} & ${\small (\alpha -1)(2\alpha -1)=\square }$ & $%
\begin{tabular}{c}
${\small (4\alpha -3)}^{2}{\small -8\beta }^{2}{\small =1}$ \\ 
${\small (4\alpha -3)}^{2}{\small +8\beta }^{2}{\small =1}$%
\end{tabular}%
$ & 
\begin{tabular}{c}
{\small Pell eq.} \\ 
{\small trivial eq.}%
\end{tabular}
\\ \hline
{\small 8} & ${\small \alpha (2\alpha -1)=\square }$ & 
\begin{tabular}{c}
${\small (2\alpha -1)}^{2}{\small -2\beta }^{2}{\small =1}$ \\ 
${\small (2\alpha -1)}^{2}{\small +2\beta }^{2}{\small =1}$%
\end{tabular}
& 
\begin{tabular}{c}
{\small Pell eq.} \\ 
{\small trivial eq.}%
\end{tabular}
\\ \hline
\begin{tabular}{c}
{\small 11, 25, 47} \\ 
{\small 48, 49}%
\end{tabular}
& ${\small \alpha -1=\beta }_{1}^{3}${\small \ }${\small \&}${\small \ }$%
{\small 2\alpha -1=\beta }_{2}^{3}$ & ${\small \beta }_{2}^{3}{\small %
+2(-\beta }_{1}^{3}{\small )=1}$ & 
\begin{tabular}{c}
{\small `classical'} \\ 
{\small equation}$^{3}$%
\end{tabular}
\\ \hline
{\small 12} & ${\small (\alpha -1)(\alpha }^{2}{\small -\alpha +1)=\square }$
& ${\small \alpha }^{3}{\small -2\alpha }^{2}{\small +2\alpha -1=\pm \beta }%
^{2}$ & {\small zero rank}$^{4}$ \\ \hline
{\small 13} & ${\small \alpha }^{2}{\small -\alpha +1=\square }$ & ${\small %
(2\alpha -1)}^{2}{\small \pm \beta }^{2}{\small =-3}$ & {\small trivial eq.}
\\ \hline
{\small 14, 20, 21} & ${\small \alpha }^{2}{\small -\alpha +1=C}$ & ${\small %
\beta }^{3}{\small -48=(8\alpha -4)}^{2}$ & {\small Ellog used}$^{5}$ \\ 
\hline
{\small 9} & ${\small \alpha (\alpha -1)(2\alpha -1)=\square }$ & 
\begin{tabular}{c}
${\small (2\alpha )}^{3}{\small -3(2\alpha )}^{2}{\small +2(2\alpha )=\beta }%
^{2}$ \\ 
${\small (-2\alpha )}^{3}{\small +3(-2\alpha )}^{2}{\small +2(-2\alpha
)=\beta }^{2}$%
\end{tabular}
& 
\begin{tabular}{c}
{\small zero rank}$^{4}$ \\ 
{\small zero rank}$^{4}$%
\end{tabular}
\\ \hline
{\small 22} & ${\small -\alpha }^{2}{\small +3\alpha -1=\square }$ & $%
{\small (2\alpha -3)}^{2}{\small \pm \beta }^{2}{\small =5}$ & {\small %
trivial eq.} \\ \hline
{\small 23} & ${\small (2\alpha -1)(-\alpha }^{2}{\small +3\alpha
-1)=\square }$ & 
\begin{tabular}{c}
${\small (-2\alpha )}^{3}{\small +7(-2\alpha )}^{2}{\small +10(-2\alpha
)+4=\beta }^{2}$ \\ 
${\small (2\alpha )}^{3}{\small -7(2\alpha )}^{2}{\small +10(2\alpha
)-4=\beta }^{2}$%
\end{tabular}
& 
\begin{tabular}{c}
{\small Ellog used}$^{5}$ \\ 
{\small zero rank}$^{4}$%
\end{tabular}
\\ \hline
\begin{tabular}{c}
{\small 24, 26,} \\ 
{\small 27, 28}%
\end{tabular}
& ${\small -\alpha }^{2}{\small +3\alpha -1=C}$ & ${\small \beta }^{3}%
{\small +80=(8\alpha -12)}^{2}$ & {\small Ellog used}$^{5}$ \\ \hline
{\small 31, 45, 46} & ${\small \alpha =\beta }_{1}^{3}${\small \ \& }$%
{\small 2\alpha -1=\beta }_{2}^{3}$ & ${\small (-\beta }_{2}{\small )}^{3}%
{\small +2\beta }_{1}^{3}{\small =1}$ & 
\begin{tabular}{c}
{\small `classical'} \\ 
{\small equation}$^{3}$%
\end{tabular}
\\ \hline
33 & $%
\begin{tabular}{c}
${\small \alpha (2\alpha -1)(2\alpha }^{2}{\small -2\alpha +1)}$ \\ 
${\small \times (3\alpha }^{2}{\small -3\alpha +1)=\square }$%
\end{tabular}%
$ & 
\begin{tabular}{c}
${\small (3\alpha )}^{3}{\small -3(3\alpha )}^{2}{\small +3(3\alpha )=\beta }%
^{2}$ \\ 
\\ 
${\small \alpha (2\alpha -1)<0}$%
\end{tabular}
& 
\begin{tabular}{c}
{\small zero rank}$^{4}$ \\ 
\\ 
{\small trivial eq.}%
\end{tabular}
\\ \hline\hline
\end{tabular}%
$\newline
\begin{equation*}
\text{\textit{continue to next page}}
\end{equation*}

\begin{equation*}
\text{\textit{continued from previous page}}
\end{equation*}%
$%
\begin{tabular}{||c|c|c|c||}
\hline\hline
\textbf{Eq. No} & 
\begin{tabular}{c}
\textbf{implies} \\ 
\textbf{or} \\ 
\textbf{is equivalent to}%
\end{tabular}
& 
\begin{tabular}{c}
\textbf{is reduced to} \\ 
\textbf{or} \\ 
\textbf{is equivalent to}%
\end{tabular}
& \textbf{Comments} \\ \hline
{\small 34} & ${\small (\alpha -1)(-2\alpha }^{2}{\small +2\alpha
-1)=\square }$ & 
\begin{tabular}{c}
${\small (-2\alpha )}^{3}{\small +4(-2\alpha )}^{2}{\small +6(-2\alpha
)+4=\beta }^{2}$ \\ 
${\small (2\alpha )}^{3}{\small -4(2\alpha )}^{2}{\small +6(2\alpha
)-4=\beta }^{2}$%
\end{tabular}
& 
\begin{tabular}{c}
{\small Ellog used}$^{5}$ \\ 
{\small zero rank}$^{4}$%
\end{tabular}
\\ \hline
{\small 32} & ${\small 3\alpha }^{2}{\small -3\alpha +1=\square }$ & 
\begin{tabular}{c}
${\small \beta }^{2}{\small -3(2\alpha -1)}^{2}{\small =1}$ \\ 
${\small \beta }^{2}{\small +3\alpha }^{2}{\small =-1}$%
\end{tabular}
& 
\begin{tabular}{c}
{\small Pell eq.} \\ 
{\small impossible}%
\end{tabular}
\\ \hline
{\small 35, 36} & ${\small \alpha (3\alpha }^{2}{\small -3\alpha +1)=\square 
}$ & 
\begin{tabular}{c}
${\small (3\alpha )}^{3}{\small -3(3\alpha )}^{2}{\small +3(3\alpha )=\beta }%
^{2}$ \\ 
${\small (-3\alpha )}^{3}{\small +3(-3\alpha )}^{2}{\small +3(-3\alpha
)=\beta }^{2}$%
\end{tabular}
& 
\begin{tabular}{l}
{\small zero rank}$^{4}$ \\ 
\multicolumn{1}{c}{{\small zero rank}$^{4}$}%
\end{tabular}
\\ \hline
{\small 37} & ${\small (2\alpha -1)(-2\alpha }^{2}{\small +2\alpha
-1)=\square }$ & 
\begin{tabular}{c}
${\small (4\alpha )}^{3}{\small -6(4\alpha )}^{2}{\small +16(4\alpha
)-16=\beta }^{2}$ \\ 
${\small (-4\alpha )}^{3}{\small +6(-4\alpha )}^{2}{\small +16(-4\alpha
)+16=\beta }^{2}$%
\end{tabular}
& 
\begin{tabular}{l}
{\small zero rank}$^{4}$ \\ 
\multicolumn{1}{c}{{\small zero rank}$^{4}$}%
\end{tabular}
\\ \hline
{\small 43, 44} & ${\small 2\alpha }^{2}{\small -2\alpha +1=C}$ & ${\small %
\beta }^{3}{\small -4=(4\alpha -2)}^{2}$ & {\small Ellog used}$^{5}$ \\ 
\hline\hline
\end{tabular}%
$

\begin{equation*}
\end{equation*}

{\Large Notes}

$^{1}${\small As a characteristic example, take equation 38. Firstly, }$%
\alpha ${\small \ can not have a common prime factor }$p${\small \ with any
of }$\alpha -1${\small , }$2\alpha -1${\small , }$2\alpha ^{2}-2\alpha +1$%
{\small , }$3\alpha ^{2}-3\alpha +1${\small . Indeed, }$\alpha \equiv 0$%
{\small \ }$(p)${\small \ implies that both }$\alpha -1${\small \ and }$%
2\alpha -1${\small \ are }$\equiv -1(p)${\small \ and both }$2\alpha
^{2}-2\alpha +1${\small \ and }$3\alpha ^{2}-3\alpha +1${\small \ are }$%
\equiv 1(p)${\small . Next, }$\alpha -1${\small \ can not have prime factor }%
$p${\small \ with any of }$2\alpha -1${\small , }$2\alpha ^{2}-2\alpha +1$%
{\small , }$3\alpha ^{2}-3\alpha +1${\small , because }$\alpha -1\equiv 0$%
{\small \ }$(p)${\small \ implies }$\alpha ${\small \ }$\equiv 1(p)${\small %
\ and, hence }$2\alpha ^{2}-2\alpha +1${\small \ and }$3\alpha ^{2}-3\alpha
+1${\small \ are both }$\equiv 1${\small \ }$(p)${\small . Analogously, }$%
2\alpha -1${\small \ can not have a common prime factor }$p${\small \ with
neither }$2\alpha ^{2}-2\alpha +1${\small \ nor }$3\alpha ^{2}-3\alpha +1$%
{\small , because, }$p${\small \ should be odd, hence }$\alpha \equiv 1/2$%
{\small \ }$(p)${\small \ and, consequently }$2\alpha ^{2}-2\alpha +1$%
{\small \ }$\equiv 1/2(p)${\small \ and }$3\alpha ^{2}-3\alpha +1\equiv
1/4(p)${\small . Finally, }$2\alpha ^{2}-2\alpha +1${\small \ and }$3\alpha
^{2}-3\alpha +1${\small \ are relatively prime because }$3(2\alpha
^{2}-2\alpha +1)-${\small \ }$2(3\alpha ^{2}-3\alpha +1)=1${\small .}

$^{2}${\small For example, although equation 44 implies that all three }$%
\alpha ${\small , }$2\alpha ^{2}-2\alpha +1${\small , }$3\alpha ^{2}-3\alpha
+1${\small \ are }$C${\small , we only use the fact that the second one is }$%
C${\small .}

{}{}$^{3}${\small The equation }$x^{3}+2y^{3}=1${\small \ has the integer
solution }$(x,y)=(-1,1)${\small , hence, by Theorem }$5${\small , Chapter 24
of \cite{MOR} can not have further solutions with }$xy\neq 0${\small .}

$^{4}${\small The only solutions are those given by coordinates of the
torsion points. These, in turn, can be computed by the Lutz-Nagell Theorem
(see, for example, Corollary 7.2, Chapter VIII.7 of \cite{JS});
automatically, they can be calculated using e.g. the PARI-}GP{\small \
calculator \cite{PAR} or the online MAGMA calculator \cite{MAG}.}

$^{5}${\small The \textit{Elliptic Logarithm Method} is applied. This has
been developed in \cite{STR} and, independently, in \cite{GEB} and now is
implemented in MAGMA \cite{MAG}; see also \cite{WBOS}.}

\subsection{The case $N=\ 4$.}

In this part we will answer the question of when a term of an EDS for which
the fourth term is zero can be a perfect square or a cube. Although the
terms of the EDSs can be a square or a cube dependent only on the parameter $%
\alpha $, there are cases when square or cube terms independent of the any
choice of $\alpha $. We first determine square or cube terms of the
sequences for which the fourth term is zero in the following theorem.

\begin{theorem}
Let $(h_{n})$ be an elliptic divisibility sequence for which the fourth term
is zero.\newline
\textit{i.} If $n\equiv 1$, $7~(8)$,$~$then $h_{n}=\square $ for all $\alpha
\in 
\mathbb{Z}
\backslash \{0\}$.\newline
ii. If $n\equiv 1$, $3$, $5$, $7~(8)$,$~$then $h_{n}=C$ for all $\alpha \in 
\mathbb{Z}
\backslash \{0\}$.
\end{theorem}

\begin{proof}
For (\textit{i}), if $n\equiv 1$ or $7~(8)$,$~$then $n=8k+1$ or $8k+7$ $%
(k\in 
\mathbb{N}
)$. Substituting these values into (\ref{8}), we have 
\begin{equation*}
h_{8k+1}=\alpha ^{24k^{2}+6k},h_{8k+7}=-\alpha ^{24k^{2}+42k+18},
\end{equation*}%
respectively. Hence, $h_{n}=\square $ for all $\alpha \in 
\mathbb{Z}
\backslash \{0\}$.\newline
For (\textit{ii}), if $n\equiv 1$, $3$, $5$ or $7~(8)$,$~$then $n=8k+1$, $%
8k+3$, $8k+5$ or $8k+7$ $(k\in 
\mathbb{N}
)$. Putting these into (\ref{8}), we have 
\begin{equation*}
h_{8k+1}=\alpha ^{24k^{2}+6k},h_{8k+3}=-\alpha
^{24k^{2}+18k+3},h_{8k+5}=\alpha ^{24k^{2}+30k+9}\text{ }
\end{equation*}%
and 
\begin{equation*}
h_{8k+7}=-\alpha ^{24k^{2}+42k+18},
\end{equation*}%
respectively. Therefore $h_{n}=C$ for all $\alpha \in 
\mathbb{Z}
\backslash \{0\}$.\newline
\end{proof}

After having determined the square or cube terms independent of the any
choice of $\alpha $, we will turn to cases when a term of an EDS for which
the fourth term is zero can be a square or a cube dependent on the parameter 
$\alpha $ in the following theorem.

\begin{theorem}
Let $(h_{n})$ be an elliptic divisibility sequence for which the fourth term
is zero.\newline
i. $h_{n}=\square $ for all $n\in 
\mathbb{N}
$ iff $\alpha =\square $. \newline
ii. $h_{n}=C$ for all $n\in 
\mathbb{N}
$ iff $\alpha =C$.
\end{theorem}

\begin{proof}
For (\textit{i})\textit{,} we have seen that if $n\equiv 1$, $7~(8)$,$~$then 
$h_{n}=\square $ for all $\alpha \in 
\mathbb{Z}
\backslash \{0\}$ in the previous theorem. Now will see in which cases a
term of $(h_{n})$ can be a perfect square. Consider the case $n\equiv 5~(8)$%
. Then $n=$ $8k+5$ for $k\in 
\mathbb{N}
$. Putting this into (\ref{8}), we have%
\begin{equation*}
h_{8k+5}=\alpha ^{24k^{2}+30k+9}\text{.}
\end{equation*}%
Hence we have $\alpha =\square $ iff $h_{8k+5}=\square $. The cases when $%
n\equiv 2,3,6~(8)$ can be proved in the similar way. So, $\alpha =\square $
iff $~h_{n}=\square $ for all $n\in 
\mathbb{N}
$.

For (\textit{ii}), we have also seen that if $n\equiv 1,3,5$ or $7~(8)$,$~$%
then $h_{n}=C$ for all $\alpha \in 
\mathbb{Z}
\backslash \{0\}$ in the previous theorem. So we must consider the cases $%
n\equiv 2$ or $6~(8)$. Then we have%
\begin{equation*}
h_{8k+2}=-\alpha ^{24k^{2}+12k+1}\text{ and }h_{8k+6}=\alpha
^{24k^{2}+36k+13}
\end{equation*}%
for $k\in 
\mathbb{N}
$, respectively, by (\ref{8}). Hence we get $\alpha =C$ iff $h_{8k+2}=C$ or $%
\alpha =C$ iff $h_{8k+6}=C$. Therefore $\alpha =C$ iff $~h_{n}=C$ for all $%
n\in 
\mathbb{N}
$.
\end{proof}

We can summarize the results which obtained above in the following table.%
\begin{equation*}
\begin{tabular}{||c|c|c|c|c|c|c|c|c||}
\hline\hline
$N=4$ & $h_{8k}$ & $h_{8k+1}$ & $h_{8k+2}$ & $h_{8k+3}$ & $h_{8k+4}$ & $%
h_{8k+5}$ & $h_{8k+6}$ & $h_{8k+7}$ \\ \hline
$\alpha \in 
\mathbb{Z}
\backslash \{0\}$ & $0$ & $\square $ &  &  & $0$ &  &  & $\square $ \\ \hline
$\alpha =\square $ & $0$ & $\square $ & $\square $ & $\square $ & $0$ & $%
\square $ & $\square $ & $\square $ \\ \hline
$\alpha \in 
\mathbb{Z}
\backslash \{0\}$ & $0$ & $C$ &  & $C$ & $0$ & $C$ &  & $C$ \\ \hline
$\alpha =C$ & $0$ & $C$ & $C$ & $C$ & $0$ & $C$ & $C$ & $C$ \\ \hline\hline
\end{tabular}%
\end{equation*}

\subsection{The Case $N=5$}

An easy calculation as in Theorem 5.1 and Theorem 5.2 gives the following
theorems.

\begin{theorem}
Let $(h_{n})$ be an elliptic divisibility sequence for which the fifth term
is zero. \newline
i. If $n\equiv 1,4,6,9~(10)$,$~$then $h_{n}=\square $ for all $\alpha \in 
\mathbb{Z}
\backslash \{0\}$. \newline
ii. If $n\equiv 1,3,4,11,12,14~(15)$,$~$then $h_{n}=C$ for all $\alpha \in 
\mathbb{Z}
\backslash \{0\}$.
\end{theorem}

\begin{theorem}
Let $(h_{n})$ be an elliptic divisibility sequence for which the fifth term
is zero.\newline
i. $h_{n}=\square $ for all $n\in 
\mathbb{N}
$ iff $\alpha =\square $. \newline
ii. $h_{n}=C$ for all $n\in 
\mathbb{N}
$ iff $\alpha =C$.
\end{theorem}

\subsection{The Case $N=6$}

This case is little more complicated than the other ones. We determine the
square or cube terms dependent on the any choice of $\alpha $ and we also
determine the square or cube terms dependent on the admissible choice of $%
\alpha $ in the following theorems. An easy calculation as in Theorem 5.1
and Theorem 5.2 gives the following theorems.

\begin{theorem}
Let $(h_{n})$ be an elliptic divisibility sequence for which the sixth term
is zero. \newline
i. If $n\equiv 1,5,7,11~(12)$,$~$then $h_{n}=\square $ for all $\alpha \in 
\mathbb{Z}
\backslash \{-1,0\}$. \newline
ii. If $n\equiv 1,3,9,15,17~(18)$,$~$then $h_{n}=C$ for all $\alpha \in 
\mathbb{Z}
\backslash \{-1,0\}$.
\end{theorem}

\begin{theorem}
Let $(h_{n})$ be an elliptic divisibility sequence for which the sixth term
is zero.\newline
i. If $n\equiv 4,8~(12)$, $\alpha +1=\square $ iff $h_{n}=\square $. \newline
ii. If $n\equiv 4,14~(18)$, $\alpha +1=C$ iff $h_{n}=C$.\newline
iii. If $n\equiv 8,10~(18)$, $\alpha =C$ iff $h_{n}=C$.
\end{theorem}

It is shown that the terms of the EDS for which the sixth term is zero can
not be a square or a cube for any choice of $\alpha $ in the following
theorem.

\begin{theorem}
Let $(h_{n})$ be an elliptic divisibility sequence for which the sixth term
is zero. \newline
i.$~$If $n\equiv 2,3,9,10~(12)$, then $h_{n}$ is not a square for all $%
\alpha \in 
\mathbb{Z}
\backslash \{-1,0\}$.\newline
ii. If $n\equiv 2,5,7,11,13,16~(18)$, then $h_{n}$ is not a cube for all $%
\alpha \in 
\mathbb{Z}
\backslash \{-1,0\}$.
\end{theorem}

\begin{proof}
For (\textit{i}), if $n\equiv 2~(12)~$then $n=12k+2$ $(k\in 
\mathbb{N}
)$. Substituting this into (\ref{n6}), we have%
\begin{equation*}
h_{12k+2}=-\alpha ^{60k^{2}+20k+1}(\alpha +1)^{48k^{2}+16k+1}\text{\ .}
\end{equation*}%
Therefore we have $h_{n}=\square $ iff%
\begin{equation}
\alpha (\alpha +1)=\square .  \tag{1}  \label{1}
\end{equation}%
This last equation leads to trivial equation%
\begin{equation*}
(2\alpha +1)^{2}\pm \beta ^{2}=1
\end{equation*}%
where $\beta $ is an integer. It is clear that the solutions of these
equations do not provide desired $\alpha $. The cases where $n\equiv
3,9,10~(12)$ can be proved in the same way.

For (\textit{ii}), if $n\equiv 2~$or $16~(18)$ then $n=18k+2$ or $n=18k+16$ $%
(k\in 
\mathbb{N}
)$. Putting these into (\ref{n6}), we have%
\begin{equation*}
h_{18k+2}=-\alpha ^{135k^{2}+30k+1}(\alpha +1)^{108k^{2}+24k+1}\text{ }
\end{equation*}%
and 
\begin{equation*}
h_{18k+16}=\alpha ^{135k^{2}+240k+106}(\alpha +1)^{108k^{2}+192k+85}
\end{equation*}%
respectively. Thus $h_{n}=C$ iff 
\begin{equation}
\alpha (\alpha +1)=C.  \tag{2}  \label{2}
\end{equation}%
If $n\equiv 5~$or $13~(18)$ then we have%
\begin{equation*}
h_{18k+5}=\alpha ^{135k^{2}+75k+10}(\alpha +1)^{108k^{2}+60k+8}\text{ }
\end{equation*}%
and 
\begin{equation*}
h_{18k+13}=\alpha ^{135k^{2}+195k+70}(\alpha +1)^{108k^{2}+156k+56}
\end{equation*}%
respectively, by (\ref{n6}). From these equations we see that $h_{n}=C$ iff 
\begin{equation}
\alpha (\alpha +1)^{2}=C.  \tag{3}  \label{3}
\end{equation}%
Now suppose that $n\equiv 7~$or $11~(18)$ then we have%
\begin{equation*}
h_{18k+7}=\alpha ^{135k^{2}+105k+20}(\alpha +1)^{108k^{2}+84k+16}\text{ }
\end{equation*}%
and 
\begin{equation*}
h_{18k+11}=-\alpha ^{135k^{2}+165k+50}(\alpha +1)^{108k^{2}+132k+40}
\end{equation*}%
respectively, by (\ref{n6}). In this case we have $h_{n}=C$ iff 
\begin{equation}
\alpha ^{2}(\alpha +1)=C.  \tag{4}  \label{4}
\end{equation}%
It follows that equations $(2)$,$(3)$ and $(4)$ lead to trivial equation%
\begin{equation*}
\beta _{2}^{3}-\beta _{1}^{3}=1
\end{equation*}%
where $\alpha =\beta _{1}^{3}$, $\alpha +1=\beta _{2}^{3}$ and $\beta _{1}$, 
$\beta _{2}$ are integers. Solutions of this equation do not provide desired 
$\alpha $.
\end{proof}

\subsection{The Case $N=7$}

An easy calculation as in Theorem 5.1 and Theorem 5.2 gives the following
theorems.

\begin{theorem}
Let $(h_{n})$ be an elliptic divisibility sequence for which the seventh
term is zero. \newline
i. If $n\equiv 1,13~(14)$,$~$then $h_{n}=\square $ for all $\alpha \in 
\mathbb{Z}
\backslash \{0,1\}$.\newline
ii. If $n\equiv 1,3,8,13,18,20~(21)$,$~$then $h_{n}=C$ for all $\alpha \in 
\mathbb{Z}
\backslash \{0,1\}$.\newline
\end{theorem}

\begin{theorem}
Let $(h_{n})$ be an elliptic divisibility sequence for which the seventh
term is zero. \newline
i. If $n\equiv 2,3,11,12~(14)$,$~$then $h_{n}=\square $ iff $\alpha
-1=\square $.\newline
ii. If $n\equiv 4,5,9,10~(14)$,$~$then $h_{n}=\square $ iff $\alpha =\square 
$.\newline
iii. If $n\equiv 4,6,10,11,15,17~(21)$,$~$then $h_{n}=C$ iff $\alpha =C$.%
\newline
iv. If $n\equiv 9,12~(21)$,$~$then $h_{n}=C$ iff $\alpha -1=C$.
\end{theorem}

\begin{theorem}
Let $(h_{n})$ be an elliptic divisibility sequence for which the seventh
term is zero.\newline
i. If $n\equiv 6,8~(14)$ then $h_{n}$ is not a square for all $\alpha \in 
\mathbb{Z}
\backslash \{0,1\}$.\newline
ii. If $n\equiv 2,5,16,19~(21)$ then $h_{n}$ is not a cube for all $\alpha
\in 
\mathbb{Z}
\backslash \{0,1\}$.
\end{theorem}

\begin{proof}
For (\textit{i}), if $n\equiv 6~$or $8~(14)$ then we have%
\begin{equation*}
h_{14k+6}=-\alpha ^{140k^{2}+120k+1}(\alpha -1)^{84k^{2}+72k+15}\text{ }
\end{equation*}%
and 
\begin{equation*}
h_{14k+8}=\alpha ^{140k^{2}+160k+45}(\alpha -1)^{84k^{2}+96k+27}
\end{equation*}%
respectively, by (\ref{n7}). From these equations we see that $h_{n}=\square 
$ iff 
\begin{equation}
\alpha (\alpha -1)=\square .  \tag{5}
\end{equation}%
This last equation leads to trivial equations%
\begin{equation*}
(2\alpha -1)^{2}\pm \beta ^{2}=1
\end{equation*}%
where $\beta $ is an integer. It is clear that the solutions of these
equations do not provide desired $\alpha $.

For (\textit{ii}), if $n\equiv 2,5,16$ or $19~(21)$ then we have%
\begin{eqnarray*}
h_{21k+2} &=&-\alpha ^{315k^{2}+60k+2}(\alpha -1)^{189k^{2}+36k+1} \\
h_{21k+5} &=&\alpha ^{315k^{2}+150k+17}(\alpha -1)^{189k^{2}+90k+10} \\
h_{21k+16} &=&-\alpha ^{315k^{2}+480k+182}(\alpha -1)^{189k^{2}+288k+109} \\
h_{21k+19} &=&\alpha ^{315k^{2}+570k+257}(\alpha -1)^{189k^{2}+342k+154}%
\text{\ }
\end{eqnarray*}%
respectively, by (\ref{n7}). Therefore, $h_{n}=C$ iff%
\begin{equation}
\alpha ^{2}(\alpha -1)=C.  \tag{6}
\end{equation}%
This last equation leads to trivial equation%
\begin{equation*}
\beta _{1}^{3}-\beta _{2}^{3}=1
\end{equation*}%
where $\alpha =\beta _{1}^{3}$, $\alpha -1=\beta _{2}^{3}$ \ and $\beta _{1}$%
, $\beta _{2}$ are integers and solutions of this equation do not provide
desired $\alpha $.
\end{proof}

\subsection{The Case $N=8$.}

We determine square or cube terms dependent on the any choice of $\alpha $
and independent of the admissible choice of $\alpha $ in the following
theorem.

\begin{theorem}
Let $(h_{n})$ be an elliptic divisibility sequence for which the eighth term
is zero. \newline
i. If $n\equiv 1,4,12,15~(16)$,$~$then $h_{n}=\square $ for all $\alpha \in 
\mathbb{Z}
\backslash \{0,1\}$. \newline
ii. If $n\equiv 3,13~(16)$,$~$then $h_{n}=\square $ iff $(\alpha -1)(2\alpha
-1)=\square $.\newline
iii. If $n\equiv 5,11~(16)$,$~$then $h_{n}=\square $ iff $\alpha (2\alpha
-1)=\square $.
\end{theorem}

\begin{proof}
For (\textit{i}), if $n\equiv 1,4,12$ or $15~(16)$,$~$then $%
n=16k+1,n=16k+4,n=16k+12,$ or $16k+15$ $(k\in 
\mathbb{N}
)$. Substituting these values into (\ref{n8}), we have%
\begin{equation*}
h_{16k+1}=\alpha ^{240k^{2}+30k}(\alpha -1)^{112k^{2}+14k}(2\alpha
-1)^{96k^{2}+12k},
\end{equation*}%
\begin{equation*}
h_{16k+4}=\alpha ^{240k^{2}+120k+4}(\alpha -1)^{112k^{2}+56k+6}(2\alpha
-1)^{96k^{2}+48k+6},
\end{equation*}%
\begin{equation*}
h_{16k+12}=-\alpha ^{240k^{2}+360k+134}(\alpha
-1)^{112k^{2}+168k+62}(2\alpha -1)^{96k^{2}+144k+54},
\end{equation*}%
and%
\begin{equation*}
\text{ }h_{16k+15}=-\alpha ^{240k^{2}+450k+134}(\alpha
-1)^{112k^{2}+210k+98}(2\alpha -1)^{96k^{2}+180k+84},
\end{equation*}%
respectively. Hence, $h_{n}=\square $ for all $\alpha \in 
\mathbb{Z}
\backslash \{0,1\}$.

For (\textit{ii}), if $n\equiv 3,13~(16)$ then we have%
\begin{equation*}
h_{16k+3}=-\alpha ^{240k^{2}+90k+8}(\alpha -1)^{112k^{2}+42k+3}(2\alpha
-1)^{96k^{2}+36k+3}\text{\ }
\end{equation*}%
and%
\begin{equation*}
\text{ \ }h_{16k+13}=\alpha ^{240k^{2}+390k+158}(\alpha
-1)^{112k^{2}+182k+73}(2\alpha -1)^{96k^{2}+156k+63}\text{\ },
\end{equation*}%
respectively by (\ref{n8}). Hence, $h_{n}=\square $ iff 
\begin{equation}
(\alpha -1)(2\alpha -1)=\square .  \tag{7}
\end{equation}%
It follows that 
\begin{equation}
(4\alpha -3)^{2}-8\beta ^{2}=1  \label{700}
\end{equation}%
or 
\begin{equation*}
(4\alpha -3)^{2}+8\beta ^{2}=1
\end{equation*}%
where $\beta $ is an integer. From the last equation we have no solutions of 
$\alpha $. The first equation leads to Pell equation. If we rewrite this
equation as%
\begin{equation*}
(\tau +2\beta \sqrt{2})(\tau -2\beta \sqrt{2})=1
\end{equation*}%
where $\tau =4\alpha -3$ we see that the only solutions of the form $\tau
_{k}+3\equiv 0$ $(4)$ give the desired solution of $\alpha $ and their
number is infinite.

For (\textit{iii}), if $n\equiv 5$ or $11~(16)$ then we have%
\begin{equation*}
h_{16k+5}=\alpha ^{240k^{2}+150k+23}(\alpha -1)^{112k^{2}+70k+10}(2\alpha
-1)^{96k^{2}+60k+9}\text{\ }
\end{equation*}%
and%
\begin{equation*}
\text{ \ }h_{16k+11}=-\alpha ^{240k^{2}+330k+113}(\alpha
-1)^{112k^{2}+154k+52}(2\alpha -1)^{96k^{2}+132k+45}\text{\ },
\end{equation*}%
respectively, by (\ref{n8}). Hence, $h_{n}=\square $ iff 
\begin{equation}
\alpha (2\alpha -1)=\square .  \tag{8}
\end{equation}%
It follows that $(\alpha ,2\alpha -1)=(\beta _{1}^{2},\beta
_{2}^{2}),(-\beta _{1}^{2},-\beta _{2}^{2}),(-\beta _{1}^{2},\beta _{2}^{2})$
or $(\beta _{1}^{2},-\beta _{2}^{2}),$ where $\beta _{1},\beta _{2}$ are
positive integers. The latter two possibilities give the trivial equations $%
2\beta _{1}^{2}+\beta _{2}^{2}=-1$ and $2\beta _{1}^{2}+\beta _{2}^{2}=1$,
respectively. The first equation is impossible and the second one does not
give desired $\alpha $. The former two possibilities lead to Pell equations%
\begin{eqnarray*}
\beta _{2}^{2}-2\beta _{1}^{2} &=&-1 \\
\beta _{2}^{2}-2\beta _{1}^{2} &=&1
\end{eqnarray*}%
respectively. The solutions to the first equation are $%
(1,1),(7,5),(41,29),...$ and parameters $\alpha $ corresponding to these
solutions are $1,25,841,...$ . Note that $\alpha $ can not be 1 by the
assumption. The solutions to the last equation are $%
(3,2),(17,12),(99,70),... $ and parameters $\alpha $ corresponding to these
solutions are $-4,-144,-4900,...$ .
\end{proof}

It is shown that the terms of the EDS for which the eighth term is zero can
not be a square for any choice of $\alpha $ in the following theorem.

\begin{theorem}
Let $(h_{n})$ be an elliptic divisibility sequence for which the eighth term
is zero. If $n\equiv 2,6,7,9,10,14~(16)$ then $h_{n}$ is not a square for
all $\alpha \in 
\mathbb{Z}
\backslash \{0,1\}$.
\end{theorem}

\begin{proof}
If $n\equiv 2,~6,~10~$or $14~(16)$ then we have%
\begin{equation*}
h_{16k+2}=-\alpha ^{240k^{2}+60k+3}(\alpha -1)^{112k^{2}+28k+1}(2\alpha
-1)^{96k^{2}+24k+1}\text{,\ }
\end{equation*}%
\begin{equation*}
h_{16k+6}=-\alpha ^{240k^{2}+180k+33}(\alpha -1)^{112k^{2}+84k+15}(2\alpha
-1)^{96k^{2}+72k+13}\text{,\ }
\end{equation*}%
\begin{equation*}
h_{16k+10}=\alpha ^{240k^{2}+420k+93}(\alpha -1)^{112k^{2}+140k+43}(2\alpha
-1)^{96k^{2}+120k+37},
\end{equation*}%
and%
\begin{equation*}
h_{16k+14}=\alpha ^{240k^{2}+420k+183}(\alpha -1)^{112k^{2}+196k+85}(2\alpha
-1)^{96k^{2}+168k+73}\text{,}
\end{equation*}%
respectively, by (\ref{n8}). From these equations we see that $h_{n}=$ $%
\square $ iff 
\begin{equation}
\alpha (\alpha -1)(2\alpha -1)=\square ,  \tag{9}  \label{9}
\end{equation}%
or equivalently $h_{n}=$ $\square $ iff 
\begin{equation*}
(2\alpha )^{3}-3(2\alpha )^{2}+2(2\alpha )=\beta ^{2}
\end{equation*}%
or 
\begin{equation*}
(-2\alpha )^{3}+3(-2\alpha )^{2}+2(-2\alpha )=\beta ^{2}.
\end{equation*}%
These last equations give elliptic curves with rank 0 and torsion points on
the these curves do not provide desired solutions of $\alpha $.

Let $n\equiv 7$, $9~(16)$. Then we have 
\begin{equation*}
h_{16k+7}=-\alpha ^{240k^{2}+210k+45}(\alpha -1)^{112k^{2}+98k+21}(2\alpha
-1)^{96k^{2}+84k+18}\text{,\ }
\end{equation*}%
and 
\begin{equation*}
h_{16k+9}=\alpha ^{240k^{2}+270k+75}(\alpha -1)^{112k^{2}+126k+35}(2\alpha
-1)^{96k^{2}+108k+30}\text{,\ }
\end{equation*}%
respectively, by (\ref{n8}). From these equations we see that $h_{n}=\square 
$ iff%
\begin{equation}
\text{ }\alpha (\alpha -1)=\square .  \tag{5}
\end{equation}%
This equation give trivial equations $(2\alpha -1)^{2}\pm \beta ^{2}=1$. It
is clear that the solutions of these equations do not provide desired
solutions of $\alpha $.
\end{proof}

We give cube terms independent of the admissible choice of $\alpha $ in the
following theorem.

\begin{theorem}
Let $(h_{n})$ be an elliptic divisibility sequence for which the eighth term
is zero. \newline
i. If $n\equiv 1,7,17,23~(24)$,$~$then $h_{n}=C$ for all $\alpha \in 
\mathbb{Z}
\backslash \{0,1\}$. \newline
ii. If $n\equiv 3,4,20,21~(24)$,$~$then $h_{n}=C$ iff $\alpha =C$.\newline
iii. If $n\equiv 6,18~(24)$,$~$then $h_{n}=C$ iff $2\alpha -1=C$.\newline
iv. If $n\equiv 9,15~(24)$,$~$then $h_{n}=C$ iff $\alpha -1=C$.
\end{theorem}

\begin{proof}
For (\textit{i}), if $n\equiv 1,7,17$ or $23~(24)$,$~$then $%
n=24k+1,n=24k+7,n=24k+17,$ or $24k+23$ $(k\in 
\mathbb{N}
)$. Substituting these values into (\ref{n8}), we have%
\begin{equation*}
h_{24k+1}=\alpha ^{540k^{2}+45k}(\alpha -1)^{252k^{2}+21k}(2\alpha
-1)^{216k^{2}+18k},
\end{equation*}%
\begin{equation*}
h_{24k+7}=-\alpha ^{540k^{2}+315k+45}(\alpha -1)^{252k^{2}+147k+21}(2\alpha
-1)^{216k^{2}+126k+18},
\end{equation*}%
\begin{equation*}
h_{24k+17}=\alpha ^{540k^{2}+765k+270}(\alpha
-1)^{252k^{2}+357k+126}(2\alpha -1)^{216k^{2}+306k+108},
\end{equation*}%
and%
\begin{equation*}
\text{ }h_{24k+23}=-\alpha ^{540k^{2}+1035k+495}(\alpha
-1)^{252k^{2}+483k+231}(2\alpha -1)^{216k^{2}+414k+198},
\end{equation*}%
respectively. Hence, $h_{n}=C$ for all $\alpha \in 
\mathbb{Z}
\backslash \{0,1\}$.

For (\textit{ii}), if $n\equiv 3,4,20$ or $21~(24)$,$~$then $%
n=24k+3,n=24k+4,n=24k+20,$ or $24k+21$ $(k\in 
\mathbb{N}
)$. Substituting these values into (\ref{n8}), we have%
\begin{equation*}
h_{24k+3}=-\alpha ^{540k^{2}+135k+8}(\alpha -1)^{252k^{2}+63k+3}(2\alpha
-1)^{216k^{2}+54k+3},
\end{equation*}%
\begin{equation*}
h_{24k+4}=-\alpha ^{540k^{2}+180k+14}(\alpha -1)^{252k^{2}+84k+6}(2\alpha
-1)^{216k^{2}+72k+6},
\end{equation*}%
\begin{equation*}
h_{24k+20}=-\alpha ^{540k^{2}+900k+374}(\alpha
-1)^{252k^{2}+420k+174}(2\alpha -1)^{216k^{2}+360k+150},
\end{equation*}%
and%
\begin{equation*}
\text{ }h_{24k+21}=\alpha ^{540k^{2}+945k+413}(\alpha
-1)^{252k^{2}+441k+192}(2\alpha -1)^{216k^{2}+378k+165},
\end{equation*}%
respectively. Thus $h_{n}=C$ iff $\alpha ^{2}=C.$ The cases (\textit{iii})
and (\textit{iv}) can be proved in the same way.
\end{proof}

It is shown that the terms of the EDS for which the eighth term is zero can
not be a cube for any choice of $\alpha $ in the following theorem.

\begin{theorem}
Let $(h_{n})$ be an elliptic divisibility sequence for which the eighth term
is zero. If $n\equiv 2$, $5$, $10$, $11$, $12$, $13$, $14$, $19$, $22~(24)$
then $h_{n}\neq C~$for all $\alpha \in 
\mathbb{Z}
\backslash \{0,1\}$.
\end{theorem}

\begin{proof}
If $n\equiv 5~(24)$ then%
\begin{equation*}
h_{24k+5}=\alpha ^{540k^{2}+225k+23}(\alpha -1)^{252k^{2}+105k+10}(2\alpha
-1)^{216k^{2}+90k+9}\text{\ }
\end{equation*}%
by (\ref{n8}). So, $h_{n}=C$ iff%
\begin{equation}
\alpha ^{2}(\alpha -1)=C.  \tag{6}
\end{equation}%
Let $n\equiv 12~(24)$. Then we have%
\begin{equation*}
h_{24k+12}=\alpha ^{540k^{2}+540k+134}(\alpha -1)^{252k^{2}+252k+62}(2\alpha
-1)^{216k^{2}+216k+54},\ 
\end{equation*}%
by (\ref{n8}). Therefore $h_{n}=C$ iff%
\begin{equation}
\alpha ^{2}(\alpha -1)^{2}=C.  \tag{10}  \label{10}
\end{equation}

The equations (6) and (10) lead to trivial equation%
\begin{equation*}
\beta _{1}^{3}-\beta _{2}^{3}=1
\end{equation*}%
where $\alpha =\beta _{1}^{3}$, $\alpha -1=\beta _{2}^{3}$ \ and $\beta _{1}$%
, $\beta _{2}$ are integers and solutions of this equation do not provide
desired $\alpha $. The cases where $n\equiv 11,13,19~(24)$ can be proved in
the same way.

If $n\equiv 2,10,14,22~(24)$ then 
\begin{equation*}
h_{24k+2}=\alpha ^{540k^{2}+90k+3}(\alpha -1)^{252k^{2}+42k+1}(2\alpha
-1)^{216k^{2}+36k+1},
\end{equation*}%
\begin{equation*}
h_{24k+10}=-\alpha ^{540k^{2}+450k+93}(\alpha -1)^{252k^{2}+210k+43}(2\alpha
-1)^{216k^{2}+180k+37},
\end{equation*}%
\begin{equation*}
h_{24k+14}=-\alpha ^{540k^{2}+630k+183}(\alpha
-1)^{252k^{2}+294k+85}(2\alpha -1)^{216k^{2}+252k+73},
\end{equation*}%
\begin{equation*}
h_{24k+22}=\alpha ^{540k^{2}+990k+453}(\alpha
-1)^{252k^{2}+462k+211}(2\alpha -1)^{216k^{2}+396k+181},
\end{equation*}%
respectively, by (\ref{n8}). From these equations we see that $h_{n}=C$ iff 
\begin{equation}
\text{ }(\alpha -1)(2\alpha -1)=C.  \tag{11}  \label{11}
\end{equation}%
It follows that $(\alpha -1,2\alpha -1)=(\beta _{1}^{3},\beta _{2}^{3})$
where $\beta _{1},\beta _{2}$ are integers. This gives the classical equation%
\begin{equation*}
\beta _{2}^{3}+2(-\beta _{1})^{3}=1
\end{equation*}%
and the only solution of this equation is $(\beta _{1}$, $\beta
_{2})=(-1,-1) $ and this solution does not provide desired $\alpha $.
\end{proof}

\subsection{The Case $N=9$}

We determine square terms independent of the admissible choice of $\alpha $.
An easy calculation as in Theorem 5.1 and Theorem 5.2 gives the following
theorems.

\begin{theorem}
Let $(h_{n})$ be an elliptic divisibility sequence for which the ninth term
is zero. \newline
i. If $n\equiv 1,17~(18)$,$~$then $h_{n}=\square $ for all $\alpha \in 
\mathbb{Z}
\backslash \{0,1\}$. \newline
ii. If $n\equiv 5,13~(18)$,$~$then $h_{n}=\square $ iff $\alpha =\square $.
\end{theorem}

It is shown that the terms of the EDS for which the ninth term is zero can
not be a square for any choice of $\alpha $ in the following theorem.

\begin{theorem}
Let $(h_{n})$ be an elliptic divisibility sequence for which the ninth term
is zero. If $n\equiv 2$, $3$, $4$, $6$, $7$, $8$, $10$, $11$, $12$, $14$, $%
15 $, $16~(18)$ then $h_{n}$ is not a square for all $\alpha \in 
\mathbb{Z}
\backslash \{0,1\}$.
\end{theorem}

\begin{proof}
If $n\equiv 6,7,8,10,11$ or $12~(18)$ then $h_{n}=\square $ iff 
\begin{equation}
\alpha (\alpha -1)=\square  \tag{5}
\end{equation}%
by (\ref{n9}). But this is impossible by proof of Theorem 5.10.

If $n\equiv 2,3,15,16~(18)$ then $h_{n}=\square $ iff 
\begin{equation}
(\alpha -1)(\alpha ^{2}-\alpha +1)=\square  \tag{12}  \label{12}
\end{equation}%
by (\ref{n9}), or equivalently $h_{n}=\square $ iff $\alpha ^{3}-2\alpha
^{2}+2\alpha -1=\pm \beta ^{2}$ where $\beta $ is an integer. These last
equations give elliptic curves with rank $0$ and torsion points on these
curves do not give desired $\alpha $.

Let $n\equiv 4,14~(18)$. Then $h_{n}=\square $ iff 
\begin{equation}
\alpha ^{2}-\alpha +1=\square  \tag{13}  \label{13}
\end{equation}%
by (\ref{n9}). This last equation leads to trivial equations 
\begin{equation*}
(2\alpha -1)^{2}-\beta ^{2}=-3
\end{equation*}%
or 
\begin{equation*}
(2\alpha -1)^{2}+\beta ^{2}=-3
\end{equation*}%
where $\beta $ is an integer. The last equation is impossible and from the
first one we only have $\alpha =0$ and $1$.
\end{proof}

We give the cube terms in the following theorem.

\begin{theorem}
Let $(h_{n})$ be an elliptic divisibility sequence for which the ninth term
is zero. \newline
i. If $n\equiv 1,3,6,12,15,21,24,26~(27)$,$~$then $h_{n}=C$ for all $\alpha
\in 
\mathbb{Z}
\backslash \{0,1\}$. \newline
ii. If $n\equiv 4,23~(27)$,$~$then $h_{n}=C$ iff $(\alpha ^{2}-\alpha
+1)^{2}=C$.
\end{theorem}

\begin{proof}
For (\textit{i}), if $n\equiv 1,3,6,12,15,21,24,26~(27)$, then an easy
calculation gives that $h_{n}=C$ for all $\alpha \in 
\mathbb{Z}
\backslash \{0,1\}$. \newline
For (\textit{ii}), if $n\equiv 4,23~(27)$,$~$then $h_{n}=C$ iff 
\begin{equation}
(\alpha ^{2}-\alpha +1)^{2}=C  \tag{14}  \label{14}
\end{equation}%
or equivalently $\alpha ^{2}-\alpha +1=$ $C$ by (\ref{n9}). This last
equation leads to an elliptic curve with rank $1$%
\begin{equation*}
\beta ^{3}-48=(8\alpha -4)^{2}
\end{equation*}%
where $\beta $ is an integer. Applying the \textit{Elliptic Logarithm Method}
we see that the only solutions to the above equation are $\alpha =-18$ and $%
19$.
\end{proof}

\begin{theorem}
Let $(h_{n})$ be an elliptic divisibility sequence for which the ninth term
is zero. If $n\equiv 2$, $5$, $7$, $8$, $10$, $11$, $13$, $14$, $16$, $17$, $%
19$, $20$, $22$, $25~(27)$ then $h_{n}\neq C~$for all $\alpha \in 
\mathbb{Z}
\backslash \{0,1\}$.
\end{theorem}

\begin{proof}
If $n\equiv 8,19$ $(27)$, then $h_{n}=C$ iff 
\begin{equation}
\alpha (\alpha -1)=C,  \tag{15}  \label{15}
\end{equation}%
if $n\equiv 10,17$ $(27)$, then $h_{n}=C$ iff 
\begin{equation}
\alpha ^{2}(\alpha -1)^{2}=C,  \tag{16}  \label{16}
\end{equation}%
if $n\equiv 2,25$ $(27)$,$~$then $h_{n}=C$ iff 
\begin{equation}
\alpha ^{2}(\alpha -1)(\alpha ^{2}-\alpha +1)=C,  \tag{17}  \label{17}
\end{equation}%
if $n\equiv $ $5,22$ $(27)$,$~$then $h_{n}=C$ iff 
\begin{equation}
\alpha (\alpha -1)(\alpha ^{2}-\alpha +1)^{2}=C,  \tag{18}  \label{18}
\end{equation}%
and if $n\equiv $ $13,14$ $(27)$,$~$then $h_{n}=C$ iff 
\begin{equation}
\alpha ^{2}(\alpha -1)^{2}(\alpha ^{2}-\alpha +1)^{2}=C.  \tag{19}
\label{19}
\end{equation}%
respectively by (\ref{n9}). These equations lead to trivial equation%
\begin{equation*}
\beta _{1}^{3}-\beta _{2}^{3}=1
\end{equation*}%
where $\alpha =$ $\beta _{1}^{3}$, $\alpha -1=\beta _{2}^{3}$ and $\beta
_{1} $, $\beta _{2}$ are integers. It is clear that the solutions of this
equation do not provide desired $\alpha $.

If $n\equiv 7,20~(27)$ then $h_{n}=C$ iff 
\begin{equation}
\alpha (\alpha ^{2}-\alpha +1)=C,  \tag{20}  \label{20}
\end{equation}%
and if $n\equiv 11,16$~$(27)$ then $h_{n}=C$ iff 
\begin{equation}
(\alpha -1)^{2}(\alpha ^{2}-\alpha +1)=C  \tag{21}  \label{21}
\end{equation}%
respectively by (\ref{n9}). These equations lead to%
\begin{equation*}
\beta ^{3}-48=(8\alpha -4)^{2}
\end{equation*}%
where $\beta $ is an integer. It is clear that the solutions of this
equation do not provide desired $\alpha $.
\end{proof}

\subsection{The Case $N=10$}

We determine square terms in the following theorem.

\begin{theorem}
Let $(h_{n})$ be an elliptic divisibility sequence for which the tenth term
is zero. \newline
i. If $n\equiv 1,9,11,19~(20)$,$~$then $h_{n}=\square $ for all $\alpha \in 
\mathbb{Z}
\backslash \{0,1\}$.\newline
ii. If $n\equiv 4,16~(20),~$then $h_{n}=\square $ iff $-\alpha ^{2}+3\alpha
-1=\square $.\newline
iii. If $n\equiv 5,15~(20),~$then $h_{n}=\square $ iff $\alpha =\square $.
\end{theorem}

\begin{proof}
The cases (\textit{i}) and (\textit{iii}) can be proved in the same way as
in Theorems 5.1 and 5.2. If $n\equiv 4,16~(20),~$then $h_{n}=\square $ iff%
\begin{equation}
-\alpha ^{2}+3\alpha -1=\square  \tag{22}  \label{22}
\end{equation}%
by (\ref{n10}). This equation leads to trivial equations%
\begin{equation*}
(2\alpha -3)^{2}\pm \beta ^{2}=5
\end{equation*}%
where $\beta $ is an integer. The only solutions of these equations are $%
\alpha =2$ and $3$.
\end{proof}

It is shown that the terms of the EDS for which the tenth term is zero can
not be a square for any choice of $\alpha $ in the following theorem.

\begin{theorem}
Let $(h_{n})$ be an elliptic divisibility sequence for which the tenth term
is zero. If $n\equiv 2,3,6,7,8,12,13,14,17,18~(20),~$then $h_{n}$ is not a
square for all $\alpha \in 
\mathbb{Z}
\backslash \{0,1\}$.
\end{theorem}

\begin{proof}
If $n\equiv 2,3,7,13,17,18~(20)$,$~$then $h_{n}=\square $ iff 
\begin{equation}
\alpha (\alpha -1)(2\alpha -1)=\square  \tag{9}
\end{equation}%
by (\ref{n10}). But this is impossible by proof of Theorem 5.12.

If $n\equiv 6,14~(20),~$then $h_{n}=\square $ iff 
\begin{equation}
\alpha (\alpha -1)=\square  \tag{5}
\end{equation}%
by (\ref{n10}), but this is impossible by proof of Theorem 5.10.

If $n\equiv 8,12~(20)$,$~$then $h_{n}=\square $ iff 
\begin{equation}
(2\alpha -1)(-\alpha ^{2}+3\alpha -1)=\square  \tag{23}  \label{23}
\end{equation}%
by (\ref{n10}). This last equation leads to 
\begin{equation*}
(2\alpha )^{3}-7(2\alpha )^{2}+10(2\alpha )-4=\beta ^{2}
\end{equation*}%
or 
\begin{equation*}
(-2\alpha )^{3}+7(-2\alpha )^{2}+10(-2\alpha )+4=\beta ^{2}
\end{equation*}%
where $\beta $ is an integer. From the first equation we have an elliptic
curve with rank zero and the torsion points on this curve do not provide
desired $\alpha $. From the last equation we have an elliptic curve with
rank 1 and applying the \textit{Elliptic Logarithm Method} we see that the
integral points on this curve do not give desired $\alpha $.
\end{proof}

We give cube terms in the following theorem.

\begin{theorem}
Let $(h_{n})$ be an elliptic divisibility sequence for which the tenth term
is zero. \newline
i. If $n\equiv 1,11,19,29~(30),~$then $h_{n}=C$.\newline
ii. If $n\equiv 3,27~(30),~$then $h_{n}=C$ iff $-\alpha ^{2}+3\alpha -1=C$.%
\newline
iii. If $n\equiv 7,13,17,23~(30),~$then $h_{n}=C$ iff $2\alpha -1=C$.
\end{theorem}

\begin{proof}
The cases (\textit{i}) and (\textit{iii}) can be proved in the same way as
in Theorems 5.1 and 5.2. If $n\equiv 3,27~(30)$,$~$then $h_{n}=C$ iff%
\begin{equation}
-\alpha ^{2}+3\alpha -1=C  \tag{24}  \label{24}
\end{equation}%
by (\ref{n10}). This equation leads to%
\begin{equation}
\beta ^{3}+80=(8\alpha -12)^{2}  \label{900}
\end{equation}%
where $\beta $ is an integer. From this equation we have an elliptic curve
with rank $1$ and applying the \textit{Elliptic Logarithm Method} we see
that the only solutions to this equation are $\alpha =2$, $3$, $38$ and $-35$%
.\qquad \newline
\end{proof}

\begin{theorem}
Let $(h_{n})$ be an elliptic divisibility sequence for which the tenth term
is zero. If $n\equiv 2$, $4$, $5$, $6$, $8$, $9$, $12$, $14$, $15$, $16$, $%
18 $, $21$, $22$, $24$, $25$, $26$, $28~(30)$ then $h_{n}\neq C~$for all $%
\alpha \in 
\mathbb{Z}
\backslash \{0,1\}$.
\end{theorem}

\begin{proof}
If $n\equiv 2,8,22,28~(30),~$then $h_{n}=C$ iff 
\begin{equation}
(\alpha -1)(2\alpha -1)(-\alpha ^{2}+3\alpha -1)=C  \tag{25}  \label{25}
\end{equation}%
by (\ref{n10}). This equation leads to classical equation%
\begin{equation*}
\beta _{2}^{3}+2(-\beta _{1}^{3})=1
\end{equation*}%
where $\alpha -1=\beta _{1}^{3}$, $2\alpha -1=\beta _{2}^{3}$ and $\beta
_{1} $, $\beta _{2}$ are integers and the only solution of this equation is $%
(\beta _{1}$, $\beta _{2})=(-1,-1)$ and this solution does not provide
desired $\alpha $.

If $n\equiv 4,14,16,26~(30),~$then $h_{n}=C$ iff%
\begin{equation}
\alpha (-\alpha ^{2}+3\alpha -1)=C\text{,}  \tag{26}  \label{26}
\end{equation}%
if $n\equiv 9,21~(30),~$then $h_{n}=C$ iff 
\begin{equation}
(2\alpha -1)^{2}(-\alpha ^{2}+3\alpha -1)=C\text{,}  \tag{27}  \label{42}
\end{equation}%
if $n\equiv 12,18~(30),~$then $h_{n}=C$ iff 
\begin{equation}
(\alpha -1)(-\alpha ^{2}+3\alpha -1)^{2}=C\text{,}  \tag{28}  \label{28}
\end{equation}%
by (\ref{n10}) or equivalently $\alpha ^{2}-3\alpha +1=C$. This last
equation leads to equation (\ref{900}). The solutions of this equation do
not provide desired $\alpha $.

If $n\equiv 5,25~(30),~$then $h_{n}=C$ iff 
\begin{equation}
\alpha (\alpha -1)(2\alpha -1)=C\text{,}  \tag{29}  \label{29}
\end{equation}%
if $n\equiv 15~(30),$ then $h_{n}=C$ iff 
\begin{equation}
\alpha (\alpha -1)(-\alpha ^{2}+3\alpha -1)=C.\newline
\tag{30}  \label{30}
\end{equation}%
by (\ref{n10}). These equations leads to trivial equation 
\begin{equation*}
\beta _{1}^{3}-\beta _{2}^{3}=1
\end{equation*}%
where $\alpha =\beta _{1}^{3}$, $\alpha -1=\beta _{2}^{3}$ and $\beta _{1}$, 
$\beta _{2}$ are integers. The solutions of this equation do not provide
desired $\alpha $.

If $n\equiv 6,24~(30),~$then $h_{n}=C$ iff 
\begin{equation}
\alpha (2\alpha -1)^{2}(-\alpha ^{2}+3\alpha -1)^{2}=C.  \tag{31}  \label{31}
\end{equation}%
This equation leads to classical equation 
\begin{equation*}
(-\beta _{2})^{3}+2\beta _{1}^{3}=1
\end{equation*}%
where $\alpha =\beta _{1}^{3}$, $2\alpha -1=\beta _{2}^{3}$ and $\beta _{1}$%
, $\beta _{2}$ are integers. The only solution of this equation is $(\beta
_{1}$, $\beta _{2})=(1,1)$ and this solution does not provide desired $%
\alpha $.
\end{proof}

\subsection{The Case $N=12$}

We determine square terms in the following theorem.

\begin{theorem}
Let $(h_{n})$ be an elliptic divisibility sequence for which the twelfth
term is zero. \newline
i. If $n\equiv 1,23~(24),~$then $h_{n}=\square $.\newline
ii. If $n\equiv 5,19~(24),~$then $h_{n}=\square $ iff $3\alpha ^{2}-3\alpha
+1=\square $.
\end{theorem}

\begin{proof}
The case (\textit{i}) can be proved in the same way as in Theorem 5.1. If $%
n\equiv 5,19~(24),~$then $h_{n}=\square $ iff%
\begin{equation}
3\alpha ^{2}-3\alpha +1=\square  \tag{32}  \label{32}
\end{equation}%
by (\ref{n12}). It follows that%
\begin{equation}
\beta ^{2}-3(2\alpha -1)^{2}=1  \label{50}
\end{equation}%
\newline
or 
\begin{equation*}
\beta ^{2}+3(2\alpha -1)^{2}=-1\text{,}
\end{equation*}%
where $\beta $ is an integer. The last equation is impossible modulo $3$,
and the first leads to Pell equation. If we rewrite equation (\ref{50}) as%
\begin{equation*}
(\beta +(2\alpha -1)\sqrt{3})(\beta -(2\alpha -1)\sqrt{3})=1
\end{equation*}%
we see that the solutions of this equation are $(2,1),(7,4),(26,15),...$ and
parameters $\alpha $ corresponding to these solutions are $1,\frac{5}{2}%
,8,...$ . It can easily be seen that the only odd solutions of this Pell
equation gives the desired solution of $\alpha $, i.e., for such an $\alpha
, $ $3\alpha ^{2}-3\alpha +1$ is a square and their number is infinite.
\end{proof}

\begin{theorem}
Let $(h_{n})$ be an elliptic divisibility sequence for which the twelfth
term is zero. If $n\equiv 2$, $3$, $4$, $6$, $7$, $8$, $9$, $10$, $11$, $13$%
, $14$, $15$, $16$, $17$, $18$, $20$, $21$, $22$ $(24)$ then $h_{n}\neq
\square ~$for all $\alpha \in 
\mathbb{Z}
\backslash \{0,1\}$.
\end{theorem}

\begin{proof}
If $n\equiv 2,3,10,14,21,22~(24),~$then $h_{n}=\square $ iff 
\begin{equation}
\alpha (2\alpha -1)(2\alpha -2\alpha ^{2}-1)(3\alpha ^{2}-3\alpha +1)=\square
\tag{33}  \label{33}
\end{equation}%
by (\ref{n12}). This equation leads to 
\begin{equation*}
(3\alpha )^{3}-3(3\alpha )^{2}+3(3\alpha )=\beta ^{2}
\end{equation*}%
\newline
or 
\begin{equation*}
\alpha (2\alpha -1)=-\beta ^{2}
\end{equation*}%
where $\beta $ is an integer. From the first equation we have an elliptic
curve with rank 0 and torsion points on this curve do not provide desired $%
\alpha $. The last equation is impossible since $\alpha $ is an integer.

If $n\equiv 4,8,16,20~(24),~$then $h_{n}=\square $ iff 
\begin{equation}
(\alpha -1)(2\alpha -2\alpha ^{2}-1)=\square  \tag{34}  \label{34}
\end{equation}%
by (\ref{n12}). This equation leads to%
\begin{equation*}
(-2\alpha )^{3}+4(-2\alpha )^{2}+6(-2\alpha )+4=\beta ^{2}
\end{equation*}%
or%
\begin{equation*}
(2\alpha )^{3}-4(2\alpha )^{2}+6(2\alpha )-4=\beta ^{2}\text{.}
\end{equation*}%
where $\beta $ is an integer. From the last equation we have an elliptic
curve with rank $0$ and torsion points on this curve do not provide desired $%
\alpha $. From the first equation we have an elliptic curve with rank $1$
and applying the \textit{Elliptic Logarithm Method} we see that the only
solutions to this equation are $\alpha =0$, $1$ and $-3$. So we have $\alpha
=-3$.

If $n\equiv 6,18~(24),~$then $h_{n}=\square $ iff 
\begin{equation}
\alpha (3\alpha ^{2}-3\alpha +1)=\square \text{,}  \tag{35}  \label{35}
\end{equation}%
if $n\equiv 11,13~(24),~$then $h_{n}=\square $ iff 
\begin{equation}
\alpha (\alpha -1)(3\alpha ^{2}-3\alpha +1)=\square .\newline
\tag{36}  \label{36}
\end{equation}%
by (\ref{n12}), or equivalently $\alpha (3\alpha ^{2}-3\alpha +1)=$ $\square 
$. This equation leads to%
\begin{equation*}
(3\alpha )^{3}-3(3\alpha )^{2}+3(3\alpha )=\beta ^{2}
\end{equation*}%
or%
\begin{equation*}
(-3\alpha )^{3}+3(-3\alpha )^{2}+3(-3\alpha )=\beta ^{2}\text{.}
\end{equation*}%
where $\beta $ is an integer. From these equations we have elliptic curves
with rank $0$ and torsion points on these curves do not provide desired $%
\alpha $.

If $n\equiv 7,17~(24),~$then $h_{n}=\square $ iff 
\begin{equation}
\alpha (\alpha -1)=\square .  \tag{5}
\end{equation}%
But this is impossible by proof of Theorem 5.10.

If $n\equiv 9,15~(24),~$then $h_{n}=\square $ iff 
\begin{equation}
(\alpha -1)(2\alpha -1)(2\alpha -2\alpha ^{2}-1)=\square .  \tag{37}
\label{37}
\end{equation}%
by (\ref{n12}). This equation leads to%
\begin{equation*}
(4\alpha )^{3}-6(4\alpha )^{2}+16(4\alpha )-16=\beta ^{2}
\end{equation*}%
or%
\begin{equation*}
(-4\alpha )^{3}+6(-4\alpha )^{2}+16(-4\alpha )+16=\beta ^{2}
\end{equation*}%
where $\beta $ is an integer. From these equations we have elliptic curves
with rank $0$ and torsion points on this curves do not provide desired $%
\alpha $.\newline
\end{proof}

An easy calculation as in Theorems 5.1 and 5.2 gives the following theorem.

\begin{theorem}
Let $(h_{n})$ be an elliptic divisibility sequence for which the twelfth
term is zero. \newline
i. If $n\equiv 1,35~(36),~$then $h_{n}=C$.\newline
ii. If $n\equiv 3,9,15,21,27,33~(36),~$then $h_{n}=C$ iff $\alpha -1=C$.
\end{theorem}

\begin{theorem}
Let $(h_{n})$ be an elliptic divisibility sequence for which the twelfth
term is zero. If $n\equiv 2$, $4$, $5$, $6$, $7$, $8$, $10$, $11$, $13$, $14$%
, $16$, $17$, $18$, $19$, $20$, $22$, $23$, $25$, $26$, $28$, $29$, $30$, $%
31 $, $32$, $34~(30)$ then $h_{n}\neq C~$for all $\alpha \in 
\mathbb{Z}
\backslash \{0,1\}$.
\end{theorem}

\begin{proof}
If $n\equiv 2,34~(36),~$then $h_{n}=C$ iff%
\begin{equation}
\alpha (2\alpha -1)(\alpha -1)^{2}(2\alpha -2\alpha ^{2}-1)(3\alpha
^{2}-3\alpha +1)=C\text{,}\newline
\tag{38}  \label{38}
\end{equation}%
if $n\equiv 8,28~(36),~$then $h_{n}=C$ iff 
\begin{equation}
\alpha (\alpha -1)^{2}(2\alpha -1)^{2}=C\text{,}  \tag{39}  \label{39}
\end{equation}%
if $n\equiv 11,25~(36),~$then $h_{n}=C$ iff 
\begin{equation}
\alpha (\alpha -1)(2\alpha -1)^{2}(2\alpha -2\alpha ^{2}-1)^{2}=C\text{,}%
\newline
\tag{40}  \label{40}
\end{equation}%
if $n\equiv 13,23~(36),~$then $h_{n}=C$ iff 
\begin{equation}
\alpha ^{2}(\alpha -1)^{2}(2\alpha -1)(2\alpha -2\alpha ^{2}-1)^{2}=C,%
\newline
\tag{41}  \label{441}
\end{equation}%
if $n\equiv 17,19~(36),~$then $h_{n}=C$ iff%
\begin{equation}
\alpha ^{2}(\alpha -1)=C  \tag{42}  \label{442}
\end{equation}%
by (\ref{n12}). These equations lead to trivial equation 
\begin{equation*}
\beta _{1}^{3}-\beta _{2}^{3}=1
\end{equation*}%
where $\alpha =\beta _{1}^{3}$, $\alpha -1=\beta _{2}^{3}$ and $\beta _{1}$, 
$\beta _{2}$ are integers. The solutions of this equation do not provide
desired $\alpha $.

If $n\equiv 4,32~(36),~$then $h_{n}=C$ iff 
\begin{equation}
(\alpha -1)(2\alpha -2\alpha ^{2}-1)^{2}=C\text{,}\newline
\tag{43}  \label{43}
\end{equation}%
if $n\equiv 14,22~(36),~$then $h_{n}=C$ iff%
\begin{equation}
\alpha ^{2}(2\alpha -2\alpha ^{2}-1)^{2}(3\alpha ^{2}-3\alpha +1)=C.\newline
\tag{44}  \label{44}
\end{equation}%
by (\ref{n12}). These equations lead to 
\begin{equation*}
\beta ^{3}-4=(4\alpha -2)^{2}
\end{equation*}%
where $\beta $ is an integer. From this equation we have an elliptic curve
with rank $1$ and applying the \textit{Elliptic Logarithm Method} we see
that the integral points on this curve do not provide desired $\alpha .$

If $n\equiv 5,31~(36),~$then $h_{n}=C$ iff 
\begin{equation}
\alpha (2\alpha -1)(2\alpha -2\alpha ^{2}-1)^{2}=C\text{,}  \tag{45}
\label{45}
\end{equation}%
if $n\equiv 16,20~(36),~$then $h_{n}=C$ iff 
\begin{equation}
\alpha ^{2}(2\alpha -1)(2\alpha -2\alpha ^{2}-1)=C\text{,}  \tag{46}
\label{46}
\end{equation}%
by (\ref{n12}). These equations lead to classical equation 
\begin{equation*}
(-\beta _{2})^{3}+2\beta _{1}^{3}=1
\end{equation*}%
where $\alpha =\beta _{1}^{3}$, $2\alpha -1=\beta _{2}^{3}$ and $\beta _{1}$%
, $\beta _{2}$ are integers. The only solution of this equation is $(\beta
_{1}$, $\beta _{2})=(1,1)$ and this solution does not provide desired $%
\alpha $.

If $n\equiv 6,18,30~(36),~$then $h_{n}=C$ iff 
\begin{equation}
(\alpha -1)^{2}(2\alpha -1)^{2}(3\alpha ^{2}-3\alpha +1)=C\text{,}\newline
\tag{47}  \label{47}
\end{equation}%
if $n\equiv 7,29~(36),~$then $h_{n}=C$ iff 
\begin{equation}
(\alpha -1)^{2}(2\alpha -1)^{2}(2\alpha -2\alpha ^{2}-1)=C\text{,}\newline
\tag{48}  \label{48}
\end{equation}%
if $n\equiv 10,26~(36),~$then $h_{n}=C$ iff 
\begin{equation}
(\alpha -1)(2\alpha -1)^{2}(3\alpha ^{2}-3\alpha +1)=C\text{,}  \tag{49}
\label{49}
\end{equation}%
by (\ref{n12}). These equations lead to classical equation 
\begin{equation*}
\beta _{2}{}^{3}+2(-\beta _{1})^{3}=1
\end{equation*}%
where $\alpha -1=\beta _{1}^{3}$, $2\alpha -1=\beta _{2}^{3}$ and $\beta
_{1} $, $\beta _{2}$ are integers. The solution of this equation is $(\beta
_{1}$, $\beta _{2})=(-1,-1)$ and this solution does not provide desired $%
\alpha $.
\end{proof}

\textbf{Acknowledgements. }\textit{The author would like to thank Joseph H.
Silverman for looking at the initial draft and making suggestions for the
first parts of the paper. Also I am grateful to Nikos Tzanakis for his
helpful suggestions for improving the paper.}

\end{document}